\documentclass[reqno,12pt]{amsart} 

\usepackage{amssymb,amsmath}

\usepackage{esint}
\usepackage{bookmark}

\usepackage[left=1.2in, right=1.2in, bottom=1.5in]{geometry} 
\newtheorem{theorem}{Theorem}[section]
\newtheorem{lemma}[theorem]{Lemma}

\newtheorem{corollary}[theorem]{Corollary}
\theoremstyle{definition}
\newtheorem{definition}[theorem]{Definition}

\theoremstyle{remark}
\newtheorem{remark}[theorem]{Remark}

\numberwithin{equation}{section}

\newcommand{\norm}[1]{\lVert#1\rVert}

\newcommand{\R}{\mathbb{R}}
\newcommand{\N}{\mathbb{N}}

\newcommand{\e}{\varepsilon}
\newcommand{\txt}[1]{\text{#1}}
\newcommand{\osc}{\text{osc}}
\newcommand{\WA}{W\!\!A}
\newcommand{\CA}{C\!A}

\begin{document}

\title[Maximal principle]{Weak maximum principle for biharmonic equations in Quasiconvex Lipschitz domains}


\author{Jinping Zhuge}
\address{Department of Math, University of Kentucky, Lexington, KY, 40506, USA.}
\curraddr{}
\email{jinping.zhuge@uky.edu}


\begin{abstract}
	In dimension two or three, the weak maximum principal for biharmonic equation is valid in any bounded Lipschitz domains. In higher dimensions (greater than three), it was only known that the weak maximum principle holds in convex domains or $C^1$ domains, and may fail in general Lipschitz domains. In this paper, we prove the weak maximum principle in higher dimensions in quasiconvex Lipschitz domains, which is a sharp condition in some sense and recovers both convex and $C^1$ domains.
\end{abstract}

\maketitle

\section{Introduction}

\subsection{Background}
In this paper, we are interested in the weak maximum principle (also known as Agmon-Miranda maximum principle) for biharmonic equations, i.e., if $\Delta^2 u = 0$ in a bounded Lipschitz domain $\Omega \subset \R^d$,
\begin{equation}\label{est.MP}
\norm{\nabla u}_{L^\infty(\Omega)} \le C\norm{\nabla u}_{L^\infty(\partial \Omega)}.
\end{equation}
We first briefly recall the history of this problem. The classical maximum principle (for second order elliptic equations) was first extended by Miranda and Agmon \cite{ADN59,A60} to biharmonic equations (or general higher order elliptic equations) in smooth domains (e.g., class of $C^4$) in any dimensions. Particularly, a related maximum principle was proved in \cite{M48} for very general domains in $\R^2$, including Lipschitz domains. In \cite{PV93}, by using the regularity theory in Lipschitz domains \cite{V90,PV92}, Pipher and Verchota established the weak maximum principle in Lipschitz domains in $\R^3$ or in $C^1$ domains for any dimensions. They also gave a counterexample that shows the maximum principle may fail for $d\ge 4$ in some Lipschitz domains containing the exterior of a cone with small aperture (the Lipschitz constant is large). The result then was extended in \cite{PV95} to polybiharmonic equations in Lipschitz domains in $\R^3$. On the other hand, the weak maximum principle for biharmonic equation in arbitrary convex domains in any dimensions was established by Shen in \cite{KS11}. Note that a convex domain must be Lipschitz but may not be of $C^1$, such as convex polyhedrons. We summarize the above known results for biharmonic equations as follows:
\begin{itemize}
	\item[(i)] For $d = 2$ or $3$, the weak maximum principle holds in arbitrary bounded Lipschitz domains;
	
	\item[(ii)] For $d \ge 4$, the weak maximum principle holds in bounded $C^1$ domains or convex domains;
	
	\item[(iii)] For $d \ge 4$, the weak maximum principle fails in some Lipschitz domain containing the exterior of a cone with small aperture.
\end{itemize}

Note that $C^1$ (smoothness) and convexity seem to be completely different geometric properties. Then, a remaining natural question arises: for $d\ge 4$, can we extend the weak maximum principle to a unified class of Lipschitz domains, covering both $C^1$ and convex domains? The purpose of this paper to give a positive answer to this question.

\subsection{Statement of main results}
In this paper, we will prove the weak maximum principle for biharmonic equations in the so-called quasiconvex Lipschitz domains for $d\ge 4$, which covers both cases in (ii). Moreover, the quaiconvexity condition is sharp in the sense that it exactly rules out the counterexamples in (iii); or in other words, the exterior of a cone with sufficiently large aperture is allowed in quasiconvex domains.

The notion of quasiconvex domains was introduced in \cite{JLW10} by Jia, Li and Wang to study the boundary regularity of the elliptic equations. Roughly speaking, a quasiconvex domain is a domain whose local boundary is close to be convex at small scales. We give an equivalent definition below which seems more natural and convenient for our application.

Let $A,B\subset \R^d$ be two non-empty sets. The Hausdorff distance between $A$ and $B$ is given by 
\begin{equation*}
d_H(A,B) = \max \{ \sup_{x\in A} \inf_{y\in B} |x-y|, \sup_{y\in B} \inf_{x\in A} |x-y| \}.
\end{equation*}
\begin{definition}[Quasiconvex domains]\label{def.quasiconvex}
	A bounded domain is said to be $(\delta,\sigma,R)$-quasiconvex if for any $0<r<R$ and $Q \in \partial\Omega$, the following conditions hold:
	\begin{itemize}
		\item[(i)] Non-degeneracy: $\Omega\cap B_r(Q)$ is connected and there exists $\sigma\in (0,1)$ so that
		\begin{equation}\label{est.reg.cond}
		|\Omega \cap B_r(Q)| \ge \sigma |B_r(Q)|.
		\end{equation}
		
		\item[(ii)] Quasiconvexity: there exists a convex domain $V = V(Q,r)$ such that $(B_r(Q)\cap \Omega) \subset V$ and $d_{H}(\partial(B_r(Q)\cap \Omega), \partial V) \le \delta r$.
	\end{itemize}
\end{definition}

\begin{remark}
	(1) The non-degeneracy condition (i) is satisfied automatically for Lipschitz domains and $\sigma$ depends only on the Lipschitz constant.
	(2) We may always assume the convex domain $V$ in condition (ii) is the convex hull of $B_r(Q)\cap \Omega$, which is the smallest convex domain containing $B_r(Q)\cap \Omega$. (3) Note that any $C^1$ ($\delta$ is arbitrarily small) or convex domains ($\delta = 0$) are quasiconvex, while a quasiconvex domain with $\delta>0$ is not necessarily $C^1$ or convex.
\end{remark}

The following is the main result of the paper.
\begin{theorem}\label{thm.MP}
	Let $\Omega$ be a bounded Lipschitz domain in $\R^d$ with $d\ge 4$. There exists $\delta_0 >0$, depending only on $d$ and the Lipschitz constant, such that the weak maximum principle (\ref{est.MP}) is true, if $\Omega$ is $(\delta,\sigma,R)$-quasiconvex with $\delta<\delta_0$.
\end{theorem}

Up to now we have not explained how we define the boundary value of the biharmonic equation, and in what sense we may understand $\nabla u$, restricted to $\partial\Omega$, on the right-hand side of (\ref{est.MP}). This question is essentially related to the solvability of Dirichlet problems in Lipschitz domains \cite{DKV86,PV91}. 

Let $\Omega$ be a bounded Lipschitz domain. For each $Q\in \partial\Omega$, there is a non-tangential ``cone'' $\Gamma(Q) = \{ x\in \Omega:  |x-Q|\le (1+\alpha)\text{dist}(x,\partial\Omega) \}$ where $\alpha>0$. If $v$ is a function in $D$, the non-tangential maximal function of $v$ is defined by
\begin{equation*}
(v)^*(Q) = \sup_{x\in \Gamma(Q)} |v(x)|, \qquad Q\in \partial\Omega.
\end{equation*}
We will use $\WA^{1,p}(\partial \Omega)$ to denote the completion of the set of arrays of functions
\begin{equation*}
\{(f,g): \phi \in C_0^\infty(\R^d), f = \phi|_{\partial\Omega}, g = \nabla \phi|_{\partial\Omega} \}
\end{equation*}
under the scale-invariant norm on $\partial\Omega$
\begin{equation*}
\norm{(f,g)}_{\WA^{1,p}(\partial\Omega)} := |\partial\Omega|^{\frac{1}{1-d}} \norm{f}_{L^p(\partial\Omega)} +  \norm{g}_{L^p(\partial\Omega)}.
\end{equation*}
Define $\WA^{2,p}(\partial\Omega) = \{(f,g)\in \WA^{1,p}(\partial\Omega): g\in W^{1,p}(\partial\Omega;\R^d) \}.$

We say that the $L^p$ Dirichlet problem, denoted by $(D)_p$, is uniquely solvable if given any $(f,g)\in \WA^{1,p}(\partial\Omega)$, there exists a unique function $u$ so that
\begin{equation}\label{eq.Dp}
\left\{
\begin{aligned}
&\Delta^2 u = 0, \quad \txt{in } \Omega, \\
& \lim_{\Gamma(Q) \ni x\to Q} u(x) = f(Q),  \quad \txt{a.e. } Q\in \partial\Omega,\\
&  \lim_{\Gamma(Q) \ni x\to Q} \nabla u(x) = g(Q), \quad \txt{a.e. } Q\in \partial\Omega.
\end{aligned}
\right.
\end{equation}
Moreover, the solution $u$ satisfies
\begin{equation}\label{est.Dp}
\norm{(\nabla u)^*}_{L^p(\partial\Omega)} \le C \norm{g}_{L^p(\partial\Omega)}.
\end{equation}
Observe that the maximum principle (\ref{est.MP}) in fact is corresponding to $(D)_\infty$ and for a.e. $Q\in \partial\Omega$, $\nabla u(Q)$ should be understood as the non-tangential limit of $\nabla u(x)$ as $\Gamma(Q)\ni x\to Q$. 

\begin{remark}\label{rmk.D2Dn}
	We mention that since $\Omega$ is a Lipschitz domain, the normal $n(Q)$ exists for a.e. $Q\in \partial\Omega$. Thus, the second boundary value condition $\nabla u(Q) = g(Q)$ in (\ref{eq.Dp}) actually may be replaced by the normal derivative $\frac{\partial}{\partial n} u(Q) = h(Q):= n(Q)\cdot g(Q)$ in the sense of non-tangential limit. This can be seen by noticing the orthogonal decomposition $g(Q) = \nabla_{\tan} f(Q) + n(Q) h(Q)$ and $|g(Q)|^2 = |\nabla_{\tan} f(Q)|^2 + |h(Q)|^2$. Hence, the tangential information of $f$ is contained in $g$ and therefore $f$ is not needed on the right-hand side of (\ref{est.Dp}). For the same reason, there is no $f$ on the right-hand side of (\ref{est.Rp}).
\end{remark}

We say that the $L^p$ regularity problem, denoted by $(R)_p$, is uniquely solvable if given any $(f,g)\in W^{2,p}(\partial\Omega)$, there exists a unique function $u$ satisfying (\ref{eq.Dp}) and
\begin{equation}\label{est.Rp}
\norm{(\nabla^2 u)^*}_{L^p(\partial\Omega)} \le C \norm{\nabla_{\tan} g}_{L^p(\partial\Omega)}.
\end{equation}

The solvability of $(D)_p$ and $(R)_p$ has been studied in many references; see, e.g., \cite{DKV86,V87,V90,PV91,PV92,DKPV97,S06,S06-2,KS11,KS11-2}. We mention a few results here. The solvability of $(D)_p$ for $2-\e < p<2+\e$ was discovered by Dahlberg, Kenig and Verchota in \cite{DKV86} for all dimensions. The left endpoint $p = 2-\e$ is sharp in the sense that for any $p<2$, there exits a Lipschitz domain so that $(D)_p$ is not solvable, even in dimension two \cite[Lemma 5.1]{DKV86}. The optimal range for $p>2$ in general Lipschitz domains is a much more complicated problem (sharp ranges are known for $d \le 7$) and is still open for $d\ge 8$ \cite{S06}.

For regularity problem, Verchota in \cite{V90} first proved the solvability of $(R)_p$ with $2-\e < p<2+\e$ for any $d\ge 2$, while the right endpoint $2+\e$ is sharp in general Lipschitz domains. Among others, Kilty and Shen then in \cite{KS11} established the duality relation between $(D)_p$ and $(R)_q$ in any Lipschitz domains, namely, $(D)_p$ is solvable if and only if $(R)_q$ is solvable for $1/p+1/q = 1, 1<p<\infty$. In particular, this property, combined with Theorem \ref{thm.MP}, leads to
\begin{corollary}
	Let $\Omega$ be a bounded Lipschitz domain in $\R^d$ with $d\ge 4$. There exists $\delta_0 >0$, depending only on $d$ and the Lipschitz constant, such that if $\Omega$ is $(\delta, \sigma,R)$-quasiconvex with $\delta < \delta_0$, then $(D)_p$ and $(R)_q$ are uniquely solvable, for all $2-\e < p \le \infty, 1<q<2+\e$.
\end{corollary}

\subsection{Main idea of proof}
The proof of Theorem \ref{thm.MP} follows the generic framework contained in \cite{PV93,S95} and takes a new generic path which may apply unifiedly to other equations or systems. For $C^1$ \cite{PV93} or convex domains \cite{KS11}, the proof of the weak maximum principle relies essentially on the solvability of $(R)_p$ for some $p>d-1$, which may fails for $d\ge 4$. However, our new path is only based on $(R)_2$, which is always true, and a reverse H\"{o}lder inequality (Calder\'{o}n-Zygmund estimate).

Heuristically, if the weak maximum principle holds for a Lipschitz domain $\Omega$, we may expect the following local $L^\infty$ property: given any $Q\in \partial\Omega, 0<r<R$, if $u\in W^{2,2}(D_r)$ is a weak solution of
\begin{equation}\label{eq.local}
\left\{
\begin{aligned}
&\Delta^2 u = 0, \quad \txt{in } D_r, \\
&u = 0,\ \nabla u = 0, \quad \txt{on } \Delta_r,
\end{aligned}
\right.
\end{equation}
where $D_r = \Omega \cap B_r(Q)$ and $\Delta_r = \partial\Omega \cap B_r(Q)$, then $|\nabla u|$ is uniformly bounded in $D_{r/2}$. In analysis, oftentimes we will ask for a slightly stronger property, namely, for any $x,y\in D_{r/2}$,
\begin{equation}\label{est.C1a}
|\nabla u(x) - \nabla u(y)| \le C \bigg( \frac{|x-y|}{r} \bigg)^{\e} \bigg( \fint_{D_{r}} |\nabla u|^2 \bigg)^{1/2}.
\end{equation}
Note that this $C^{1,\e}$ continuity can be compared to the boundary De Giorgi-Nash estimate for the second-order elliptic equations. In this paper, we will develop a general scheme to show that the property (\ref{est.C1a}), together with $(R)_2$ regularity, implies the weak maximum principle. The main idea of this scheme is to make use of the Poisson integral for biharmonic functions
\begin{equation}\label{eq.Poisson}
\begin{aligned}
u(x) & = \int_{\partial \Omega} \Delta_Q G(Q,x) \frac{\partial}{\partial n}u(Q) d\sigma(Q)  \\
&\qquad + \int_{\partial \Omega}\frac{\partial}{\partial n} \Delta_Q G(Q,x) u(Q) d\sigma(Q),
\end{aligned}
\end{equation}
where $G$ is the Green's function, and combine (\ref{est.C1a}) and $(R)_2$ to estimate the Green's function near the boundary. Some classical techniques from \cite{PV93,S95} will also be critical to get rid of the higher-order derivatives of $G$.
We point out that, with suitable adjustments, the above scheme may be applied to all dimensions (see Remark \ref{rmk.d23}) and to elliptic systems, Stokes systems\footnote{The behavior of the biharmonic functions is very like the solutions of Stokes systems. The weak maximum principle for Stokes system in Lipschitz domains in $\R^3$ was proved in \cite{S95}. But this problem is quite open for $d\ge 4$, as far as we know, even no counterexamples are known. Our approach in this paper may be used to show the weak maximum principle for Stokes system with $d\ge 4$ in Lipschitz domains with sufficiently small Lipschitz constant. }, etc.

The problem now is reduced to the property (\ref{est.C1a}). In view of the Sobolev embedding theorem, it is sufficient to show the following boundary reverse H\"{o}lder inequality for some $p>d$
\begin{equation}\label{est.reverse}
\bigg( \fint_{D_{r/2}} |\nabla^2 u|^p \bigg)^{1/p} \le C\bigg( \fint_{D_r} |\nabla^2 u|^2 \bigg)^{1/2},
\end{equation}
where $u\in W^{2,2}(D_r)$ is the weak solution of (\ref{eq.local}). For lower dimensions ($d = 2,3$), this can be shown in any Lipschitz domains; see Remark \ref{rmk.d23}. For higher dimensions ($d\ge 4$), (\ref{est.reverse}) is not true for $p>d$ in general Lipschitz domains. Therefore, some additional geometric condition for the domain is necessary. And it turns out that the quasiconvexity is exactly the right condition (for (\ref{est.reverse}), the domain does not need to be Lipschitz).


In order to prove (\ref{est.reverse}) for some $p>d$ in quasiconvex domains, we develop a new boundary perturbation approach, of independent interest, using Meyers' estimate (see Theorem \ref{thm.Meyer}) and a real variable method of Shen \cite{S05,S06-2,S07,S18} (see Theorem \ref{thm.RealVar} for a simplified version). The philosophy of the perturbation approach is as follows: (1) First, Mayboroda and Maz'ya proved in \cite{MM08} that in any convex domains, the Hessian $|\nabla^2 u|$ of a solution of (\ref{eq.local}) is uniformly bounded (see Theorem \ref{thm.Hessian.Bound}). (2) Then it is possible to derive the weaker estimate (\ref{est.reverse}) when the boundary of a domain can be locally viewed as perturbations of convex domains at all small scales. We mention that another well-known boundary perturbation approach applied to Reifenberg flat domains, originating from Caffarelli and Peral \cite{CP98}, was developed by Wang and Byun in \cite{BW04} (more application may be found in, e.g., \cite{BW05,BW08,BW08-2,BW10,MP12}), which uses compactness and a proof of contradiction (as a result, the dependence of constants cannot be specified). Their approach was also used in \cite{JLW10,JLW11,BKSW15} in the setting of quasiconvex domains. Our new approach in this paper, quite different from theirs, is straightforward and do not use compactness or the proof of contradiction. Therefore, all the constants in the estimates can be computed explicitly, if necessary. In particular, with a careful examination, one can actually know how $\delta_0$ in Theorem \ref{thm.MP} depends on the parameters, and how small it needs to be. We would also like to say a few words about Shen's real variable method as a key ingredient in our proof. This method was also inspired by \cite{CP98} and can be viewed as a refined and dual version of the Calder\'{o}n-Zygmund lemma \cite{S18}. Previouly, it has been used to deal with problems in different motivations. In this paper, for the first time, we show that Shen's real variable method is also a powerful (simple, well-organized and quantitative) tool for boundary perturbation problems and we expect many applications of our new approach for regularity theory in non-smooth domains.

Finally, as we have indicated, the property (\ref{est.C1a}) is slightly stronger than the boundness of $|\nabla u|$, which means it might implies the $C^{1,\alpha}$ continuity of the solution in appropriate setting. Actually, we will show that under the same condition as Theorem \ref{thm.MP}, a biharmonic function is indeed in $C^{1,\alpha}(\overline{\Omega})$ ($0\le \alpha<\e, C^1 = C^{1,0}$), i.e., $\nabla u$ is $C^{\alpha}$-H\"{o}lder continuous up to the boundary in the classical way, if the boundary value is $C^{1,\alpha}$ in proper sense. Again, the approach we use here can be carried out for lower-dimensional cases in arbitrary Lipschitz domains, as well as elliptic systems, Stokes systems, etc.

The organization of the paper is as follows. In Section 2, we introduce the weak solutions, Caccioppoli inequality and Meyers' estimate. In Section 3, we prove the reverse H\"{o}lder inequality in quasiconvex domains. The pointwise estimates of Green's function is established in Section 4. The main result, Theorem \ref{thm.MP}, is proved in Section 5. Finally, the classical solutions in $C^{1,\alpha}(\overline{\Omega})$ is obtained in Section 6.
\section{Preliminaries}

\subsection{Weak solutions}
We first give a proper interpretation for the weak solution of (\ref{eq.local}). Let $\Omega$ be a bounded domain and $Q\in \partial \Omega$. Let $D_r = D_r(Q) = \Omega\cap B_r(Q)$ and $\Delta_r = \Delta_r(Q) = \partial\Omega \cap B_r(Q)$ (these notations will be used throughout). We say $u\in W^{2,2}(D_r)$ is a weak solution of
\begin{equation}\label{eq.weaksol}
\left\{
\begin{aligned}
&\Delta^2 u = 0, \quad \txt{in } D_r \\
&u = 0,\ \nabla u = 0, \quad \txt{on } \Delta_r,
\end{aligned}
\right.
\end{equation}
if for any $\phi \in C_0^\infty(D_r)$,
\begin{equation*}
\int_{D_r} \Delta u \Delta \phi = 0,
\end{equation*}
and for any $\psi\in C_0^\infty(B_r(Q))$, $u\psi \in W_0^{2,2}(D_r)$. Recall that for any open set $E$, $W_0^{2,2}(E)$ is the closure of $C_0^\infty(E)$ under the norm of $W^{2,2}(E)$.

By the above definition, it is not difficult to see that the solution $u$ can be extended to a function $\widetilde{u}$ in $W^{2,2}(B_r(Q))$ by zero-extension, namely
\begin{equation*}
\widetilde{u}(x) = 
\left\{
\begin{aligned}
&u(x), \qquad &x&\in D_r, \\
&0, \qquad & x&\in B_r(Q)\setminus D_r.
\end{aligned}
\right.
\end{equation*}
This property will be useful for us.

\subsection{Caccioppoli inequality}
For the Caccioppoli inequality and the following Meyer's estimate, we require that the domain $\Omega$ satisfies the exterior non-degeneracy condition: there exists a constant $c>0$ such that for any $Q\in \partial\Omega$ and $|B_r(Q)\setminus \Omega| \ge c|B_r(Q)|$.

Before we proceed, we first show that quasiconvex domains always satisfy the exterior non-degeneracy condition.

\begin{lemma}
	A $(\delta,\sigma,R)$-Quasiconvex domain ($\delta<1$) must satisfy the exterior non-degeneracy condition.
\end{lemma}
\begin{proof}
	Let $\Omega$ be a $(\delta,\sigma,R)$-Quasiconvex domain with $\delta<1$. By definition, for any $Q\in \partial\Omega$ and $0<r<R$, there exists a convex domain $V$ so that $d_H(\Omega\cap B_r(Q), \partial V) \le \delta r$. Hence, there exists a point $z\in \partial V$ so that $|z-Q|\le \delta r$ and thus $\text{dist}(z,\partial B_r(Q)) \ge (1-\delta)r$. Because $V$ is convex, we can find a ``tangent'' plane at $z\in \partial V$, so that $V$ lies in one side of the ``tangent'' plane, namely, $V\subset \{ x\in \R^d: n\cdot (x-z) > 0 \}$ for some unit vector $n$. Since $\Omega\cap B_r(Q) \subset V$, we have $B_r(Q)\setminus \Omega \supset B_r(Q)\cap \{x\in \R^d: n\cdot (x-z) < 0 \}$. Now, the desired estimate $|B_r(Q)\setminus \Omega| \ge c|B_r(Q)|$ follows from a simple geometrical observation and the facts $z\in B_r(Q)$ and $\text{dist}(z,\partial B_r(Q)) \ge (1-\delta)r$.
\end{proof}

\begin{theorem}[Caccioppoli inequality] \label{lem.Caccioppoli}
	Let $\Omega\subset \R^d$ satisfy the exterior non-degeneracy condition. Let $r \in (0,\text{\em diam}(\Omega))$ and $Q\in \partial\Omega$. Suppose $u\in W^{2,2}(D_r)$ is a weak solution of (\ref{eq.weaksol}).
	Then,
	\begin{equation}\label{est.Caccioppoli}
	\int_{D_{r/4}(Q)} |\nabla^2 u|^2 \le  \frac{C}{\rho^2} \int_{D_{r/2}(Q)} |\nabla u|^2 \le \frac{C}{\rho^4} \int_{D_{r}(Q)} |u|^2,
	\end{equation}
	where $C$ depends only on the dimension $d$ and $c$ only.
\end{theorem}
This theorem has been proved, for example, in \cite[Corollary 23]{B16}. The exterior non-degeneracy condition is only necessary for the first inequality of (\ref{est.Caccioppoli}), in order to use the Poincar\'{e} inequality.

\subsection{Meyers' estimate}
The Meyers' estimate for second order elliptic equations are well-known. In this paper, we will use a version of Meyer's estimate for biharmonic functions.
\begin{theorem}[Meyers' estimate]\label{thm.Meyer}
	Let $\Omega$ satisfy the exterior non-degeneracy condition. Let $u \in W^{2,2}(D_r)$ be the weak solution of (\ref{eq.weaksol}).
	Then there exists some $p_0>2$, depending only on $d$ and $c$, so that
	\begin{equation*}
	\bigg( r^{-d} \int_{D_{r/2}} |\nabla^2 u|^{p_0} \bigg)^{1/{p_0}} \le C\bigg(r^{-d} \int_{D_r} |\nabla^2 u|^2 \bigg)^{1/2},
	\end{equation*}
	where $C$ depends only on $d$ and $c$.
\end{theorem}

\begin{proof}
	A more general version may be found in \cite[Theorem 24]{B16}. We give an outline of the proof for the readers' convenience. First of all, extend $u$ to a function in $W^{2,2}(B_r)$ by zero-extension, which will still denoted by $u$. Let $Q\in \Delta_{r/2}$ and $\rho \in (0,cr)$. By the Caccioppoli inequality, one has
	\begin{equation}\label{est.reverse.BQ}
	\begin{aligned}
	\bigg( \fint_{B_\rho(Q)} |\nabla^2 u|^2 \bigg)^{1/2} & \le \frac{C}{\rho^4} \bigg( \fint_{B_{4\rho}(Q)} |u|^2 \bigg)^{1/2} \\
	& \le \bigg( \frac{C\rho^2 |B_{4\rho}(Q)|}{|B_{4\rho}(Q) \setminus D_r|} \bigg)^2 \bigg( \fint_{B_{4\rho}(Q)} |\nabla^2 u|^q \bigg)^{1/q}\\
	& \le C \bigg( \fint_{B_{4\rho}(Q)} |\nabla^2 u|^q \bigg)^{1/q},
	\end{aligned}
	\end{equation}
	where
	\begin{equation*}
	\frac{1}{q} =\min \Big\{ \frac{1}{2} + \frac{2}{d},1 \Big\},
	\end{equation*}
	and in the second ienquality we also used (twice) a Riesz potential representation \cite[Lemma 7.16]{GT01}
	\begin{equation*}
	|u(x)| \le \frac{C\rho^2 |B_{4\rho}|}{|B_{4\rho} \setminus D_r|} \int_{B_{4\rho}} \frac{|\nabla u(x)|}{|x-y|^{d-1}}dy \qquad \text{for } x\in B_{4\rho},
	\end{equation*}
	and the well-known Hardy-Littlewood-Sobolev inequality. In view of the zero-extension, the estimate actually holds for any $Q\in B_r$. Hence, by a generalized Gehring's inequality (i.e., the self-improving property of reverse H\"{o}lder inequality; see \cite{G73} or \cite[Chapter 12, Theorem 4.1]{CW98}), there exists some $p_0>2$ depending only on $d$ and $c$ so that
	\begin{equation*}
	\bigg( \fint_{B_{r/2}} |\nabla^2 u|^{p_0} \bigg)^{1/{p_0}} \le C\bigg(\fint_{B_r} |\nabla^2 u|^2 \bigg)^{1/2}.
	\end{equation*}
	Finally, replacing $B_r$ by $D_r$ leads to the desired estimate.
\end{proof}

\section{Reverse H\"{o}lder inequality}
In this section, we are interested in the boundary reverse H\"{o}lder inequality for the Hessian
\begin{equation*}
\bigg( \fint_{D_{r/2}} |\nabla^2 u|^p \bigg)^{1/p} \le C\bigg( \fint_{D_r} |\nabla^2 u|^2 \bigg)^{1/2}.
\end{equation*}
Note that Theorem \ref{thm.Meyer} gives such estimate for some $p_0>2$ close to $2$. We will improve this estimate to large $p$ (particularly for $p>d$), provided the domain is $(\delta,\sigma,R)$-quasiconvex with small $\delta>0$. To apply a boundary perturbation approach, we need the following a priori estimate proved by Mayboroda and Maz'ya.

\begin{theorem}[\cite{MM08}]\label{thm.Hessian.Bound}
	Let $\Omega$ be a convex domain in $\R^d$. Fix some $r\in (0,\text{\em diam}(\Omega))$ and $Q\in \partial\Omega$. Suppose $u$ is a weak solution of (\ref{eq.weaksol}) in $D_r = D_{r}(Q)$. Then
	\begin{equation*}
	|\nabla^2 u(x)| \le \frac{C}{r^2 } \bigg( \fint_{D_r} |u|^2 \bigg)^{1/2}, \qquad \text{for any } x\in D_{r/2},
	\end{equation*}
	where $C$ depends only on the dimension $d$ only.
\end{theorem}

Our boundary perturbation approach is based on Shen's real variable argument. The following is a  simplified version of \cite[Theorem 4.2.3]{S18}.

\begin{theorem}[Shen's real variable argument]\label{thm.RealVar}
	Let $B_0$ be a ball in $\R^d$ and $F\in L^2(4B_0)$. Let $q>2$. Suppose that for each ball $B \subset 2B_0$ with $|B|\le c_0|B_0|$, there exist two measurable functions $F_B$ and $R_B$ on $2B$ such that $|F|\le |F_B| + |R_B|$ on $2B$, and
	\begin{equation}\label{est.RealCond}
	\begin{aligned}
	\bigg( \fint_{2B} |R_B|^q \bigg)^{1/q} & \le C_1 \bigg( \fint_{4B} |F|^2 \bigg)^{1/2},\\
	\bigg( \fint_{2B} |F_B|^2 \bigg)^{1/2} &\le \eta \bigg( \fint_{4B} |F|^2 \bigg)^{1/2},
	\end{aligned}
	\end{equation}
	where $C_1 >1$ and $0<c_0<1$. Then for any $2<p<q$ there exists $\eta_0>0$, depending only on $C_1,c_0,p,q$, with the property that if $0\le \eta \le \eta_0$, then $F\in L^p(B_0)$ and
	\begin{equation*}
	\bigg( \fint_{B_0} |F|^p \bigg)^{1/p} \le C\bigg( \fint_{4B_0} |F|^2 \bigg)^{1/2},
	\end{equation*}
	where $C$ depends at most on $C_1,c_0,p$ and $q$.
\end{theorem}

We point out that the balls $4B$ and $2B$ in (\ref{est.RealCond}) are not essential and may be replaced by $\alpha B$ for any fixed $\alpha>0$.

We also need the following lemmas regarding geometric properties of the quasiconvex domains.

\begin{lemma}\label{lem.LipConst}
	Let $V$ be a convex domain contained in $B_1 = B_1(0)$ and $|V| \ge \sigma |B_1|$. Then $V$ is a Lipschitz domain with a constant depending only on $d$ and $\sigma$.
\end{lemma}
\begin{proof}
	First of all, $V$ is a Lipschitz domain since it is convex. Since $V$ is contained in $B_1$, we have
	\begin{equation*}
	\mathbb{H}^{d-1}(\partial V) \le \mathbb{H}^{d-1}(\partial B_1).
	\end{equation*}
	This can be shown by approximating $\partial V$ by convex polyhedrons and taking the limit. Now, consider the set $V_t = \{ x\in V: \text{dist}(x,\partial V) < t \}$. Again, by an approximation argument, we see
	\begin{equation*}
	|V_t| \le t \mathbb{H}^{d-1}(\partial V) \le t \mathbb{H}^{d-1}(\partial B_1).
	\end{equation*}
	If
	\begin{equation*}
	t < r_0: = \frac{\sigma |B_1|}{\mathbb{H}^{d-1}(\partial B_1)},
	\end{equation*}
	then $|V_t| < |V|$ and thus, $V\setminus V_t \neq \emptyset.$ This implies that there exists some point $x_0 \in V$ so that $B_{r_0}(x_0) \subset V$.
	
	Now, by the convexity of $V$, for each $Q \in \partial V$, the cone connecting $Q$ and $B_{r_0}(x_0)$,
	\begin{equation*}
	\mathcal{C}_Q := \{ (1-t)Q+t x: x\in B_{r_0}(x_0), 0<t<1 \}
	\end{equation*}
	is contained in $V$, since $Q$ is on the boundary and $B_{r_0}(x_0)$ is contained strictly in $V$. Note that all these cones $\{ \mathcal{C}_Q: Q\in \partial V\}$ have uniform lower bounds of height and aperture comparable to $r_0$. This implies that the Lipschitz constant of $\partial V$ comparable to $r_0$. Actually, if we consider the localized boundary $\partial V \cap B_{r_0/4}(P)$ for some $P\in \partial V$, there exists a fixed cone $\mathcal{C}_0 = C_0(P)$ whose vertex is the origin, axis is parallel to $\overrightarrow{Px_0}$ and aperture is smaller but still comparable to $r_0$, so that $Q+ \mathcal{C}_0 \subset V$ for all $Q\in \partial V \cap B_{r_0/4}(P)$. This shows that the Lipschitz constant of $\partial V \cap B_{r_0/4}(P)$ is comparable to $r_0$.
\end{proof}

\begin{lemma}\label{lem.flatness}
	Let $\Omega$ be a $(\delta,\sigma,R)$-quaiconvex domain. Let $Q\in \partial \Omega, r\in (0,R/4), B_r = B_r(Q)$ and $V_{4r}$ be the convex hull of $\Omega\cap B_{4r}(Q)$. Then for $0<t<1$,
	\begin{equation}\label{est.flat}
	\begin{aligned}
	&\{ x\in \Omega: \text{\em dist}(x,\partial\Omega) < tr \} \cap B_r \\
	& \qquad \subset W_{r,t}:= \{ x\in V_{4r}: \text{\rm dist}(x,\partial V_{4r} \cap B_{3r}) \le (t+\delta)r \}.
	\end{aligned}
	\end{equation}
	Moreover, $|W_{r,t}| \le C(t+\delta) r^d$, where $C$ depends only on $d$ and $\sigma$.
\end{lemma}

\begin{proof}
	Let $w\in \{ x\in \Omega: \text{dist}(x,\partial\Omega) < tr \} \cap B_r$. Then, there exists a point $y\in \partial\Omega$ so that $|w-y|< tr$. Thus, $|y-Q| \le |w-Q| + |w-y| < r+tr < 2r$. Hence $y\in B_{2r}(Q) \cap \partial\Omega$. Now, by the definition of $V_{4r}$ and the $(\delta,\sigma,R)$-quasiconvexity, there exists a point $z\in \partial V_{4r}$ so that $|y-z| \le \delta r$. Since $\delta \in (0,1)$, $|z-Q| \le |y-Q| + |z-y| < 3r$ and $|w-z| \le |w-y| + |y-z| < tr + \delta r = (t+\delta) r$. This implies $z\in \partial V_{4r} \cap B_{3r}(Q)$ and therefore
	\begin{equation*}
	w \in \{ x\in V_{4r}: \text{dist}(\partial V_{4r} \cap B_{3r} ) \le (t+\delta )r \}.
	\end{equation*}
	This gives (\ref{est.flat}) since $x$ is an arbitrary point in $\{ x\in \Omega: \text{dist}(x,\partial\Omega) < tr \} \cap B_r$.
	
	To estimate $|W_{r,t}|$, note that $V_{4r} \subset B_{4r}$ is a convex set and $|V_{4r}| \ge \sigma |B_{4r}|$, by the quasiconvexity. By rescaling and Lemma \ref{lem.LipConst}, the Lipschitz constant of $r^{-1} V_{4r}$ depends only on $d$ and $\sigma$. This actually implies $|W_{r,t}| \le C(t+\delta) r \mathbb{H}^{d-1}(\partial V_{4r}) \le C(t+\delta) r^d$.
\end{proof}

%

\begin{lemma}\label{lem.newPoincare}
	Let $\Omega$ be a $(\delta,\sigma,R)$-quaiconvex domain. Let $Q\in \partial \Omega, r\in (0,R/4), B_r = B_r(Q)$ and $V_{4r}$ be the convex hull of $B_{4r} \cap \Omega$. Suppose $u\in W^{1,2}(B_{4r})$ and $u = 0$ on $B_{4r} \setminus V_{4r}$. Then, if $0<t+\delta<1$,
	\begin{equation}\label{est.newPoincare}
	\int_{B_r \cap \Omega_{tr}} u^2 \le C(t+\delta)^2 r^2 \int_{W_{t,r}} |\nabla u|^2,
	\end{equation}
	where $\Omega_{tr} = \{ x\in \Omega: \text{\rm dist}(x,\partial\Omega) < tr \}$, $W_{t,r}$ is given as in (\ref{est.flat}) and $C$ depends only on $d$ and $\sigma$.
\end{lemma}
\begin{proof}
	Note that Lemma \ref{lem.flatness} implies $B_r \cap \Omega_{tr} \subset W_{r,t}$. By our assumption, we see $u = 0$ on $\partial V_{4r} \cap B_{4r}$. Then, (\ref{est.newPoincare}) follows from the Poincar\'{e} inequality in $W_{r,t}$, i.e.,
	\begin{equation*}
	\int_{W_{r,t}} u^2 \le C(t+\delta)^2 r^2 \int_{W_{r,t}} |\nabla u|^2,
	\end{equation*}
	where $C$ depends only on $d$ and $\sigma$. We point out that the last inequality is valid because the Lipschitz constant of $V_{4r}$ depends only on $d$ and $\sigma$ (after rescaling), and $W_{r,t}$ is a boundary layer with thickness $(t+\delta) r$.
\end{proof}

The following is the main theorem of this section.
\begin{theorem}\label{thm.Reverse Holder}
	Let $d\ge 4$ and $\Omega$ be a bounded $(\delta,\sigma,R)$-quasiconvex domain. Then for any $p\in (2,\infty)$, there exists $\delta_0>0$, depending only on $d,p$ and $\sigma$, such that if $\delta\in  (0,\delta_0)$, $r\in (0,R)$, and $u \in W^{2,2}(D_{2r})$ is the weak solution of
	\begin{equation}\label{est.weak.D2r}
	\left\{
	\begin{aligned}
	&\Delta^2 u = 0, \quad \txt{in } D_{2r}, \\
	&u = 0,\ \nabla u = 0, \quad \txt{on } \Delta_{2r},
	\end{aligned}
	\right.
	\end{equation}
	then
	\begin{equation}\label{est.reverseLp}
	\bigg( \fint_{D_{r/2}} |\nabla^2 u|^p \bigg)^{1/p} \le C\bigg( \fint_{D_{2r}} |\nabla^2 u|^2 \bigg)^{1/2},
	\end{equation}
	where $C$ depends only on $d$ and $\Omega$\footnote{Here and after, for convenience, we simply say $C$ depends on $\Omega$ if $C$ depends on $(\delta,\sigma,R)$ and/or the Lipschitz constant.}.
\end{theorem}

\begin{proof}
	Since $u$ and $\nabla u$ vanish on $\Delta_{2r} = \partial\Omega \cap B_{2r}$, we may first extend $u\in W^{2,2}(D_{2r})$ to $\tilde{u} \in W^{2,2}(B_{2r})$ by the zero-extension. Fix $\rho\in (0,r/16)$ and $D_\rho = D_\rho(Q)\subset D_r = D_r(Q)$. Since $\Omega$ is $(\delta,\sigma,R)$-quasiconvex, by Definition \ref{def.quasiconvex}, the convex hull of $D_\rho$, denoted by $V_\rho$, is a convex domain so that
	\begin{equation*}
	V_\rho \cap \Omega = D_\rho \quad \txt{and} \quad |d_H(\partial V_\rho, \partial D_\rho)| \le \delta \rho.
	\end{equation*}
	We now construct an approximation of $\tilde{u}$ in $V_\rho$. Let $v$ be the weak solution of
	\begin{equation*}
	\left\{
	\begin{aligned}
	&\Delta^2 v = 0, \quad \txt{in } V_\rho, \\
	&v = \tilde{u},\ \nabla v = \nabla \tilde{u}, \quad \txt{on } \partial V_\rho.
	\end{aligned}
	\right.
	\end{equation*}
	Note that $\tilde{u} = 0$ and $\nabla \tilde{u} = 0$ on $\partial V_\rho\cap B_\rho$. 
	
	We claim that
	\begin{equation}\label{est.claim1}
	\norm{\nabla^2 v}_{L^\infty(D_{\rho/2})} \le C \bigg( \fint_{D_{\rho}} |\nabla^2 u|^2 \bigg)^{1/2},
	\end{equation}
	and
	\begin{equation}\label{est.claim2}
	\bigg( \fint_{D_{\rho}} |\nabla^2 (v-u)|^2 \bigg)^{1/2} \le C\delta^{ \e } \bigg( \fint_{D_{8\rho}} |\nabla^2 u|^{2} \bigg)^{1/2},
	\end{equation}
	for some $\e>0$. 
	
	Due to the zero-extension, in the above claim, $D_\rho$ can be replaced by $B_\rho$ in the above estimates. Also, it may be applied to $B_\rho(Q')$ for any $Q'\in B_{r/2}(Q)$ and $\rho\le r/2$. Thanks to Theorem \ref{thm.RealVar}, this implies the desired estimate (\ref{est.reverseLp}). Actually, for any given $p\in (2,\infty)$, choose $\delta_0>0$ such that $C\delta_0^{\e}< \eta_0$. Then for any $\delta\in (0,\delta_0)$,we apply Theorem \ref{thm.RealVar} to $F = |\nabla^2 u|, R_B = |\nabla^2 v|,$ and $R_B = |\nabla^2(v- u)|$ and obtain (\ref{est.reverseLp}).
	
	It suffices to show the claims (\ref{est.claim1}) and (\ref{est.claim2}). First of all, (\ref{est.claim1}) follows from Theorem \ref{thm.Hessian.Bound} applied to $v$ in the convex domain $V_\rho$. Actually,
	\begin{equation*}
	\begin{aligned}
	\norm{\nabla^2 v}_{L^\infty(D_{\rho/2})} &\le \norm{\nabla^2 v}_{L^\infty(V_{\rho/2})} \\
	& \le C\bigg( \fint_{V_\rho} |\nabla^2 v|^2 \bigg)^{1/2} \\
	& \le C\bigg( \fint_{D_\rho} |\nabla^2 u|^2 \bigg)^{1/2},
	\end{aligned}
	\end{equation*}
	where we have used the energy estimate in the last inequality and the fact $|V_\rho|\simeq |D_\rho| \simeq \rho^d$ (due to the non-degeneracy condition (\ref{est.reg.cond})). Here and after, we say $A\simeq B$ if there exist positive constants $c$ and $C$ (depending only on the parameters of the domain) so that $cB \le A \le CB$.
	
	To see (\ref{est.claim2}), by the integration by parts, we have
	\begin{equation}\label{eq.DeltaU-V}
	\begin{aligned}
	\int_{V_\rho} |\Delta (v-\tilde{u})|^2 & = -\int_{V_\rho} \Delta \tilde{u} \Delta (v - \tilde{u}) \\
	& = -\int_{D_\rho} \Delta u \Delta (v - u).
	\end{aligned}
	\end{equation}
	Let $\Omega_t = \{ x\in \Omega: \text{dist}(x,\partial\Omega) < t \}$. Let $\theta_{\delta \rho}$ be a cut-off function such that $\theta_{\delta \rho} = 1$ in $\Omega\setminus \Omega_{2\delta \rho}$ and $\theta_{\delta \rho} = 0$ in $\Omega_{\delta \rho}$, and $|\nabla^k \theta_{\delta \rho}| \le C(\delta \rho)^{-k} $. Write
	\begin{equation}\label{est.DeltaUL2}
	\int_{V_\rho} |\Delta (v-\tilde{u})|^2 = -\int_{D_\rho} \Delta (\theta_{\delta \rho} u) \Delta (v - u) - \int_{D_\rho} \Delta ((1-\theta_{\delta \rho}) u) \Delta (v - u).
	\end{equation}
	
	By the integration by parts and the fact $\Delta^2 u =0$ in $D_\rho$, we write the first integral of (\ref{est.DeltaUL2}) as
	\begin{equation}\label{est.DeltaU1}
	\begin{aligned}
	& \int_{D_\rho} \Delta (\theta_{\delta \rho} u) \Delta (v - u) \\
	&\qquad = \int_{D_\rho} \Delta \theta_{\delta \rho} u \Delta(v-u) + 2\int_{D_\rho} \nabla \theta_{\delta \rho}\cdot \nabla u \Delta(v-u) \\
	& \qquad\qquad -\int_{D_\rho} \Delta \theta_{\delta \rho} \Delta u (v-u) - 2\int_{D_\rho} \nabla \theta_{\delta \rho} \cdot \nabla(v-u) \Delta u.
	\end{aligned}
	\end{equation}
	Note that $\nabla \theta_{\delta \rho}$ is supported in $\Omega_{2\delta \rho}$, and Lemma \ref{lem.flatness} implies
	\begin{equation*}
	D_\rho \cap \Omega_{2\delta \rho} \subset W_{\rho,2\delta}:= \{ x\in V_{4\rho}: \text{dist}(x,\partial V_{4\rho} \cap B_{3\rho}) < 3\delta \rho \}
	\end{equation*}
	and $|W_{\rho,2\delta}| \le C\delta \rho^d$.
	 Hence, using the vanishing boundary conditions, Lemma \ref{lem.newPoincare} and the Meyers' estimate, we have
	\begin{equation}\label{est.uL2}
	\begin{aligned}
	\bigg( \int_{D_\rho \cap \Omega_{2\delta \rho}} |u|^2 \bigg)^{1/2} + \delta \rho \bigg( \int_{D_\rho \cap \Omega_{2\delta \rho}} |\nabla u|^2 \bigg)^{1/2} &\le  C\delta \rho \bigg( \int_{W_{\rho,2\delta}} |\nabla u|^2 \bigg)^{1/2}\\
	& \le C(\delta \rho)^2 \bigg( \int_{W_{\rho,2\delta}} |\nabla^2 u|^2 \bigg)^{1/2} \\
	& \le C(\delta \rho)^{2}(\delta \rho^d)^\e \bigg( \int_{D_{4\rho}} |\nabla^2 u|^{p_0} \bigg)^{1/p_0},
	\end{aligned}
	\end{equation}
	where $\e = \frac{1}{2} - \frac{1}{p_0}>0$. In the last inequality, we also use the facts $W_{\rho,2\delta} \subset B_{4\rho}$ and $\nabla^2 u = 0$ in $B_{4\rho}\setminus D_{4\rho}$.
	
	Similarly, since $v-u$ and $\nabla(v-u)$ vanish on $\partial V_\rho$, and $D_\rho \cap \Omega_{2\delta\rho} \subset U_\rho:= \{x\in V_{\rho}: \text{dist}(x,\partial V_\rho) \le 3\delta \rho \}$, we may apply the Poincar\'{e} inequality on the layer $U_\rho$ to obtain
	\begin{equation}\label{est.uvL2}
	\begin{aligned}
	&\bigg( \int_{D_\rho \cap \Omega_{2\delta \rho} } |v - \tilde{u}|^2 \bigg)^{1/2} + \delta \rho\bigg( \int_{D_\rho \cap \Omega_{2\delta \rho}} |\nabla( v - \tilde{u})|^2 \bigg)^{1/2} \\
	&\qquad \le C(\delta \rho)^2 \bigg( \int_{ U_\rho} |\nabla^2 (v-\tilde{u})|^{2} \bigg)^{1/2} \\
	&\qquad \le C(\delta \rho)^2 \bigg( \int_{V_{\rho}} |\nabla^2 (v-\tilde{u})|^{2} \bigg)^{1/2},
	\end{aligned}
	\end{equation}
	where $C$ depends only on $d$ and $\sigma$, due to Lemma \ref{lem.LipConst}.
	
	Combining (\ref{est.uL2}) and (\ref{est.uvL2}), we obtain from (\ref{est.DeltaU1}) that
	\begin{equation}\label{est.DeltaUV}
	\begin{aligned}
	&\bigg| \int_{D_\rho} \Delta (\theta_{\delta \rho} u) \Delta (v - u) \bigg| \\
	&\quad \le C(\delta \rho^d)^{\e } \bigg( \int_{D_{4\rho}} |\nabla^2 u|^{p_0} \bigg)^{1/p_0} \bigg( \int_{V_\rho} |\nabla^2 (v-u)|^{2} \bigg)^{1/2}.
	\end{aligned}
	\end{equation}
	
	Similarly, the second integral in the right-hand side of (\ref{est.DeltaUL2}) has the same bound as (\ref{est.DeltaUV}). It follows that
	\begin{equation}\label{est.uVr}
	\begin{aligned}
	&\int_{V_\rho} |\Delta (v-\tilde{u})|^2 \\
	&\quad \le C(\delta \rho^d)^{\e } \bigg( \int_{D_{4\rho}} |\nabla^2 u|^{p_0} \bigg)^{1/p_0} \bigg( \int_{V_\rho} |\nabla^2 (v-u)|^{2} \bigg)^{\frac{1}{2}}.
	\end{aligned}
	\end{equation}
	
	Note that
	\begin{equation*}
	\int_{V_\rho} |\Delta (v-\tilde{u})|^2 = \int_{V_\rho} |\nabla^2 (v-\tilde{u})|^2.
	\end{equation*}
	
	Consequently, we obtain by the Meyers' estimate
	\begin{equation*}
	\begin{aligned}
	\bigg( \fint_{D_\rho} |\nabla^2 (v-\tilde{u})|^2 \bigg)^{1/2} &\le C\delta^{\e } \bigg( \fint_{D_{4\rho}} |\nabla^2 u|^{p_0} \bigg)^{1/p_0} \\
	& \le C\delta^{\e } \bigg( \fint_{D_{8\rho}} |\nabla^2 u|^{2} \bigg)^{1/2},
	\end{aligned}
	\end{equation*}
	which proves (\ref{est.claim2}) and hence completes the proof.
\end{proof}

\begin{corollary}\label{cor.DuCe}
	Let $d\ge 4$ and $\Omega$ be a bounded $(\delta,\sigma,R)$-quasiconvex domain. Then for any given $\e \in (0,1)$, there exists $\delta_0>0$, depending only on $d,\e$ and $\sigma$, such that if $\delta\in  (0,\delta_0)$, $r\in (0,R)$, and $u \in W^{2,2}(D_{2r})$ is the weak solution of (\ref{est.weak.D2r}), then for any $x,y\in D_r$
	\begin{equation}\label{est.Du.Holder}
	|\nabla u(x) - \nabla u(y)| \le C \bigg( \frac{|x-y|}{r} \bigg)^{\e} \bigg( \fint_{D_{2r}} |\nabla u|^2 \bigg)^{1/2},
	\end{equation}
	where $C$ depends only on $d,\e$ and $\Omega$.
\end{corollary}
\begin{proof}
	This follows readily from Theorem \ref{thm.Reverse Holder} (with $p = d/(1-\e)$), Theorem \ref{lem.Caccioppoli} and the Sobolev embedding theorem.
\end{proof}

\begin{remark}\label{rmk.d23}
	As mentioned in the introduction, the reverse H\"{o}lder inequality (\ref{est.reverseLp}) with $q>d$ is critical for us. We emphasize here that this can be shown in general Lipschitz domains for lower dimensions $d=2$ or $3$.  
	For $d = 2$, the Meyers' estimate gives (\ref{est.reverseLp}) with $p = 2+\e$. For $d = 3$, the Meyers' estimate alone is not enough, while the sharp $(R)_{2+\e}$ regularity in Lipschitz domains is also needed. To see this, let $u\in W^{2,2}(D_{2r})$ be the weak solution of (\ref{est.weak.D2r}). Without loss of generality, assume $D_{2r}$ is a Lipschitz domain. By the coarea formula and the Meyers estimate, there exists $t\in (1,2)$ so that
	\begin{equation}\label{est.Meyers.3d}
	\bigg( \fint_{\partial D_{tr} \cap \Omega } |\nabla^2 u|^{2+\e} d\sigma \bigg)^{1/(2+\e)} \le C \bigg( \fint_{D_{2r}} |\nabla^2 u|^2 \bigg)^{1/2}. 
	\end{equation}
	Due to the fact that $u =0 $ and $\nabla u = 0$ on $\Delta_{2r} = \partial\Omega \cap B_{2r}$, we use the $(R)_2$ regularity to obtain
	\begin{equation}\label{est.R2.3d}
	\norm{(\nabla^2 u)^*}_{L^{2+\e}(\partial D_{tr})} \le C\norm{\nabla^2 u}_{L^{2+\e}(\partial D_{tr} \cap \Omega)} \le Cr^{\frac{2}{2+\e} - \frac{3}{2}} \norm{\nabla^2 u}_{L^2(D_{2r})}.
	\end{equation}
	Now, recall a general inequality: $\norm{F}_{L^q(D_r)} \le C\norm{(F)^*}_{L^p(\partial D_r)}$ with $q = dp/(d-1)$; see e.g., \cite[Remark 9.3]{KLS13}. Hence, if $d =3$, combining (\ref{est.Meyers.3d}) and (\ref{est.R2.3d}) gives
	\begin{equation*}
	\norm{\nabla^2 u}_{L^{3+\frac{3\e}{2}}(D_{tr})} \le Cr^{\frac{2}{2+\e} - \frac{2}{3}} \norm{\nabla^2 u}_{L^2(D_{2r})}.
	\end{equation*}
	This implies
	\begin{equation*}
	\bigg( \fint_{D_{tr}} |\nabla^2 u|^{3+\frac{3\e}{2}}\bigg)^{1/(3+\frac{3\e}{2})} \le C\bigg( \fint_{D_{2r}} |\nabla^2 u|^2 \bigg)^{1/2},
	\end{equation*}
	which yields the desired estimate as $t\in (1,2)$.
\end{remark}

\section{Green's function}
This section is devoted to constructing the Green's function in quasiconvex domains and obtaining some global pointwise estimates.

For each $y\in \R^d$, let $\Gamma^y(x) = \Gamma(x,y)$ be the fundamental solution of $\Delta^2$ in $\R^d$ \cite{S07}:
\begin{equation*}
\Gamma(x,y) =  \left\{
\begin{aligned}
&\frac{1}{8\pi} |x-y|^2 (\ln|x-y|-1), \quad &d=2, \\
& \frac{-1}{2\omega_3} |x-y|, \quad &d =3, \\
& \frac{-1}{4\omega_4} \ln |x-y|, \quad & d = 4, \\
&  \frac{1}{2(n-2)(n-4)\omega_d|x-y|^{d-4}},\quad & d\ge 5,
\end{aligned}
\right.
\end{equation*}
where $\omega_d$ is the surface area of the $d$-dimensional unit sphere.

To construct the Green's function in $\Omega$, we assume $\Omega$ is a bounded quasiconvex domain. For any $y\in \Omega$, let $\gamma(x,y)$ be the solution of
\begin{equation}\label{eq.gamma}
\left\{
\begin{aligned}
&\Delta_x^2 \gamma(x,y) = 0, \quad \txt{in } \Omega, \\
&\gamma(x,y) = \Gamma(x,y),\ \nabla_x \gamma(x,y) = \nabla_x \Gamma(x,y), \quad \txt{on } \partial\Omega.
\end{aligned}
\right.
\end{equation}
Since $\Gamma(\cdot,y)$ is smooth in $\R^d\setminus \{y \}$, the variational solution of (\ref{eq.gamma}) exists and $\gamma(\cdot,y) \in W^{2,2}(\Omega)$. Define the Green's function
\begin{equation*}
G(x,y) = \Gamma(x,y) - \gamma(x,y), \qquad x,y\in \Omega.
\end{equation*}
Then, $G(x,y)$ satisfies
\begin{equation*}
\left\{
\begin{aligned}
&\Delta_x^2 G(x,y) = \delta_y(x), \quad \txt{in } \Omega, \\
&G(x,y) = 0,\ \nabla_x G(x,y) = 0, \quad \txt{on } \partial\Omega,
\end{aligned}
\right.
\end{equation*}
and $G(\cdot, y)\in W^{2,2}(\Omega\setminus B_r(y) )$ for any $r>0$. Moreover, by a standard argument, one may show the symmetry property:
\begin{equation*}
G(x,y) = G(y,x), \qquad \text{for any } x,y\in \Omega, x\neq y.
\end{equation*}


\begin{lemma}\label{lem.EstG}
	Let $d\ge 4$ and $\Omega$ be a $(\delta,\sigma,R)$-quasiconvex Lipschitz domain. There exists $\delta_0>0$, depending only on $d$ and $\sigma$, so that the Green's function $G(x,y)$ satisfies, for any $x,y\in \Omega$,
	\begin{equation}\label{est.G}
	|G(x,y)| \le \left\{
	\begin{aligned}
	&C\ln \frac{d_\Omega}{|x-y|}, \quad &d = 4, \\
	&\frac{C}{|x-y|^{d-4}}, \quad &d \ge 5,
	\end{aligned}
	\right.
	\end{equation}
	
	\begin{equation}\label{est.DxG.DyG}
	|\nabla_x G(x,y)| + |\nabla_y G(x,y)| \le \frac{C}{|x-y|^{d-3}},
	\end{equation}
	and
	\begin{equation}\label{est.DxDyG}
	|\nabla_x \nabla_y G(x,y)| \le \frac{C}{|x-y|^{d-2}}.
	\end{equation}
	In (\ref{est.G}), $d_\Omega$ is the diameter of $\Omega$.
\end{lemma}

\begin{proof}
	Fix $x_0, y_0\in \Omega$ and $r_0 = \frac{1}{2}|x_0 - y_0|$. 
	Let $L$ be the straight line through $x_0$ and $y_0$, and $\bar{x} = (x_0+y_0)/2$. Let $x_k$ be points on $L$ such that $|x_k - \bar{x}| = \frac{5}{3}|x_{k-1} - \bar{x}|$ for any $k = 1,2,\cdots$, and $|x_k - x_0|<|x_k - y_0|$. Then, it is not hard to see $r_k := |x_k - \bar{x}| = (\frac{5}{3})^{k} r_0$. Note that there are only finite points $\{ x_k \}$ contained in $\Omega \cap L$. Moreover, $\{ B_{r_k/4} (x_k) \}$ forms a sequence of balls connecting $x_0$ to the boundary $\partial\Omega$.  Similarly, we may let $y_k$ be the points on $L$ such that $|y_k -\bar{x}| = (\frac{5}{3})^{k} r_0$ and $|y_k - y_0|<|y_k-x_0|$. 
	
	Let $M, N$ be the smallest natural numbers such that the ball $B_{r_M/4}(x_M)$ and $B_{r_N/4}(y_N)$ intersect $\partial\Omega$.
	
	We will use a duality argument. Let $f_k = (f_{k,ij}) \in C_0^\infty(B_{r_k}(y_k)\cap \Omega;\R^{d\times d})$ and $u\in W_0^{2,2}(\Omega)$ be the weak solution of 
	\begin{equation*}
	\left\{
	\begin{aligned}
	&\Delta^2 u = \nabla^2\cdot f_k, \quad \txt{in } \Omega, \\
	&u = 0,\ \nabla u = 0, \quad \txt{on } \partial\Omega.
	\end{aligned}
	\right.
	\end{equation*}
	Integrating the equation against $u$ and by the integration by parts, we have
	\begin{equation*}
	\int_{\Omega} |\nabla^2 u|^2 
	= \int_{\Omega} |\Delta u|^2 = \int_{B_{r_k}(y_k)} f_k\cdot \nabla^2 u,
	\end{equation*}	
	which yields
	\begin{equation}\label{est.Du2.f}
	\bigg( \int_{\Omega} |\nabla^2 u|^2 \bigg)^{1/2} \le C \bigg( \int_{B_{r_k}(y_k)} |f_k|^2 \bigg)^{1/2}.
	\end{equation}
	It follows that, for any $\ell\ge 0$,
	\begin{equation}\label{est.D2u.fk}
	\bigg( \fint_{B_{r_\ell}(x_\ell)} |\nabla^2 u|^2 \bigg)^{1/2} \le Cr_\ell^{-d/2} \bigg( \int_{B_{r_k}(y_k)} |f_k|^2 \bigg)^{1/2}.
	\end{equation}
	Observe that $B_{r_\ell}(x_\ell) \cap \txt{supp}(f_k) = \emptyset$. By Corollary \ref{cor.DuCe}, there exists $\delta_0>0$, depending only on $d$ and $\sigma$, so that if $\delta<\delta_0$,
	\begin{equation*}
	\begin{aligned}
	\underset{B_{r_\ell/4}(x_\ell)}{\osc} [\nabla u] &\le Cr_\ell \bigg( \fint_{B_{r_\ell}(x_\ell)} |\nabla^2 u|^2 \bigg)^{1/2} \\
	& \le  Cr_\ell^{1-d/2} \bigg( \int_{B_{r_k}(y_k)} |f_k|^2 \bigg)^{1/2},
	\end{aligned}
	\end{equation*}
	where $\underset{U}{\osc} [F] = \sup_{x,y\in U} |F(x) - F(y)|$ and we have used the Poincar\'{e} inequality and (\ref{est.D2u.fk}). By noting that $\nabla u(x) = 0$ on $\partial\Omega$ and $x_0$ is connected to the boundary through a sequence of balls $B_{r_\ell/4}(x_\ell)$ with $1\le \ell \le N$, we know that the above estimate actually implies 
	\begin{equation}\label{est.Du.fL2}
	\begin{aligned}
	|\nabla u(x_0)| &\le \sum_{\ell = 1}^{N} Cr_\ell^{1-d/2} \bigg( \int_{B_{r_k}(y_k)} |f_k|^2 \bigg)^{1/2} \\
	&\le Cr_0^{1-d/2} \bigg( \int_{B_{r_k}(y_k)} |f_k|^2 \bigg)^{1/2}.
	\end{aligned}
	\end{equation}

	Now, recall the representation formula
	\begin{equation*}
	\nabla u(x) = \int_{\Omega} \nabla_x \nabla_y^2 G(x,y) \cdot f_k(y)dy.
	\end{equation*}
	Thus, (\ref{est.Du.fL2}) implies
	\begin{equation*}
	\bigg| \int_{\Omega} \nabla_x \nabla_y^2 G(x_0,y)  \cdot f_k(y)dy. \bigg| \le \frac{C}{r_0^{d/2-1}} \bigg( \int_{B_{r_k}(y_k)} |f_k|^2 \bigg)^{1/2}.
	\end{equation*}
	By duality, it follows that for any $x,w\in B_{r_k/4}(x_k)$
	\begin{equation}\label{est.DxDDyG}
	\bigg( \fint_{B_{r_k}(y_k)} |\nabla_x \nabla_y^2 G(x_0,y)|^2 dy \bigg)^{1/2} \le \frac{C}{r_0^{d/2-1} r_k^{d/2} }.
	\end{equation}
	
	Note that $\nabla_x G(x,\cdot)=\nabla_x G(\cdot,x)$ is biharmonic in $\Omega\setminus \{x \}$ with vanishing boundary condition. Again, by (\ref{est.Du.Holder}), (\ref{est.DxDDyG}) and the reverse H\"{o}lder inequality, we have
	\begin{equation*}
	\underset{B_{r_k/4}(y_k)}{\osc} [\nabla_x \nabla_y G(x_0,\cdot)] \le \frac{C}{r_0^{d/2-1} r_k^{d/2-1} }.
	\end{equation*}
	Due to the boundary condition $\nabla_x \nabla_y G(x_0, y) = 0$ for $y\in \partial\Omega$, by connecting $y_0$ to the boundary through a sequence of $B_{r_k/4}(y_k)$, the above condition implies
	\begin{equation*}
	| \nabla_x \nabla_y G(x_0 ,y_0) | \le
	\sum_{k=1}^{M} \frac{C}{r_0^{d/2-1} r_k^{d/2-1}} \le
	\frac{C}{r_0^{d - 2}} = \frac{C}{|x_0 - y_0|^{d-2}},
	\end{equation*}
	which yields (\ref{est.DxDyG}). Then (\ref{est.DxG.DyG}) and (\ref{est.G}) follows easily by the fundamental theorem of calculus and the vanishing boundary conditions of $G(x,y)$.
\end{proof}

\begin{theorem}\label{thm.Gxy}
	Let $d\ge 4$ and $\Omega$ be a bounded $(\delta,\sigma,R)$-quasiconvex domain. Then for any given $\e \in (0,1)$, there exists $\delta_0>0$, depending only on $d,\e$ and $\sigma$, such that if $\delta\in  (0,\delta_0)$, the Green's function satisfies
	
	\begin{equation}\label{est.G.Holder}
	| G(x,y)| \le \frac{C\delta(x)^{1+ \e} \delta(y)^{1+\e} }{|x-y|^{d-2+2\e}},
	\end{equation}
	
	\begin{equation}\label{est.DxG.Holder}
	|\nabla_x G(x,y)| \le \frac{C\delta(x)^\e \delta(y)^{1+\e} }{|x-y|^{d-2+2\e}},
	\end{equation}
	
	\begin{equation}\label{est.DyG.Holder}
	|\nabla_y G(x,y)| \le \frac{C\delta(x)^{1+\e} \delta(y)^{\e} }{|x-y|^{d-2+2\e}},
	\end{equation}
	and
	\begin{equation}\label{est.DxDy.Holder}
	|\nabla_x \nabla_y G(x,y)| \le \frac{C\delta(x)^\e \delta(y)^\e }{|x-y|^{d-2+2\e}},
	\end{equation}
	where $\delta(x) = \text{\rm dist}(x,\partial\Omega)$ and $C$ depends only on $d,\e$ and $\Omega$.
\end{theorem}

\begin{proof}
	We show (\ref{est.DxDy.Holder}) first. Observe that $H(x,\cdot) := \nabla_x G(x,\cdot)$ is a solution in $\Omega\setminus \{x\}$. Fix $x,y\in \Omega$ and $r = |x-y|$. Without loss of generality, assume $\delta(y) < \frac{1}{3} r$ and let $Q_y\in \partial\Omega$ be a point that $\delta(y) = |y - Q_y|$. We apply (\ref{est.Du.Holder}) and (\ref{est.DxDyG}) to $H(x,\cdot)$ in $D_{2r}(Q_y)$ and obtain
	\begin{equation*}
	\begin{aligned}
	&|\nabla_y H(x,y) - \nabla_y H(x,Q_y)| \\
	&\qquad \le C\bigg( \frac{|y-Q_y|}{r} \bigg)^\e \bigg( \fint_{D_{2r}(Q_y)} |\nabla_z H(x,z)|^2 dz \bigg)^{1/2} \\
	& \qquad \le C \frac{\delta(y)^\e }{|x-y|^{d-2+\e} }.
	\end{aligned}
	\end{equation*}
	Recall that $\nabla_y H(x,Q_y) = 0$. Hence,
	\begin{equation}\label{est.DxDyG.dy}
	|\nabla_y\nabla_x G(x,y)| \le C \frac{\delta(y)^\e }{|x-y|^{d-2+\e} }.
	\end{equation}
	
	Now, consider $H_1(\cdot,y)$ = $\nabla_y G(\cdot, y)$. Using the same argument as before and (\ref{est.DxDyG.dy}), we have
	\begin{equation*}
	|\nabla_x H_1(x,y)| \le C \bigg( \frac{\delta(x)}{|x-y|} \bigg)^\e \frac{\delta(y)^\e }{|x-y|^{d-2+\e} },
	\end{equation*}
	which implies (\ref{est.DxDy.Holder}).
	
	Next, (\ref{est.DxG.Holder}) and (\ref{est.DyG.Holder}) follow by integrating (\ref{est.DxDy.Holder}) in $y$ or in $x$, respectively, and using the boundary condition $\nabla_x G(x,\cdot) = \nabla_y G(\cdot,y) = 0$ on $\partial\Omega$. Finally, (\ref{est.G.Holder}) follows readily by integrating (\ref{est.DxG.Holder}) in $x$ and using the fact $G(\cdot, y) = 0$ on $\partial\Omega$.
\end{proof}

\section{Maximal principle}
Let $(f,g) \in \WA^{1,\infty}(\partial\Omega)$. In view of Remark \ref{rmk.D2Dn}, we may rewrite the equation (\ref{eq.Dp}) as
\begin{equation}\label{eq.bi.fg}
\left\{
\begin{aligned}
&\Delta^2 u = 0, \quad \txt{in } \Omega, \\
&u = f,\  \frac{\partial}{\partial n}u = h, \quad \txt{on } \partial\Omega,
\end{aligned}
\right.
\end{equation}
where $h = n\cdot g$ on $\partial\Omega$. Due to the $(D)_2$ solvability in Lipschitz domains \cite{DKV86}, we know there is a unique solution of (\ref{eq.bi.fg}) so that 
\begin{equation*}
\norm{(\nabla u)^*}_{L^2(\partial\Omega)} \le C\norm{g}_{L^2(\partial\Omega)} \simeq C\big( \norm{\nabla_{\tan} f}_{L^2(\partial\Omega)}  +\norm{h}_{L^2(\partial\Omega)} \big).
\end{equation*}

The purpose of this section is to prove Theorem \ref{thm.MP}. Precisely, we would like to show the solution of (\ref{eq.bi.fg}) satisfies
\begin{equation}\label{est.Du.Linfty}
\norm{\nabla u}_{L^\infty(\Omega)} \le C\norm{g}_{L^\infty(\partial\Omega)} \simeq C\big( \norm{\nabla_{\tan} f}_{L^\infty(\partial\Omega)}  +\norm{h}_{L^\infty(\partial\Omega)} \big).
\end{equation}
The starting point of the proof is the integral representation for the solution of (\ref{eq.bi.fg}) in terms of the Green's function
\begin{equation}\label{eq.IntRepre}
u(x) = \int_{\partial \Omega} \frac{\partial}{\partial n} \Delta_Q G(Q,x) f(Q) d\sigma(Q) - \int_{\partial \Omega} \Delta_Q G(Q,x) h(Q) d\sigma(Q).
\end{equation}
Since $u$ is biharmonic in $\Omega$, the interior estimate implies
\begin{equation*}
|\nabla u(x)| \le \frac{C \sup \{ |u(y)|: y\in B_{\delta(x)/2}(x) \} }{\delta(x)}.
\end{equation*}
Consequently, to show (\ref{est.Du.Linfty}), it suffices to show the pointwise estimate of $u(x)$ for $x\in \Omega$. This will be done by considering the two integrals in (\ref{eq.IntRepre}) separately. For convenience, throughout this section, we fix a small absolute value $\e \in (0,1)$ and let $\delta_0 >0$ be given by Corollary \ref{cor.DuCe} (or Theorem \ref{thm.Gxy}). Since $\Omega$ is a Lipschitz domain and $\sigma$ can be determined in terms of the Lipschitz constant, $\delta_0$ actually depends only on $d$ and the Lipschitz constant.

\begin{lemma}\label{lem.DeltaG}
	Let $d\ge 4$ and $\Omega$ be a bounded $(\delta,\sigma,R)$-quasiconvex Lipschitz domain with $\delta<\delta_0$. Then
	\begin{equation*}
	\int_{\partial\Omega} |\Delta_Q G(Q,x)| d\sigma(Q) \le C\delta(x),
	\end{equation*}
	where $C$ depends only on $d$ and $\Omega$.
\end{lemma}

\begin{proof}
	Let $\{ B_{k\ell} = B_{2^k \delta(x)}(Q_{k\ell}) : 1\le \ell \le N, k\ge 0, Q_{k\ell}\in \partial\Omega\}$ be a sequence of balls that cover $\partial\Omega$ with finite overlaps, where $N$ is a finite number depending only on dimension. Moreover,
	\begin{equation*}
	\txt{dist}(x, 4B_{k\ell} ) \simeq 2^k \delta(x).
	\end{equation*}
	Let $D_{k\ell} = B_{k\ell}\cap \Omega$ and $\Delta_{k\ell} = B_{k\ell}\cap \partial\Omega$. For each $k$ and $\ell$, we estimate
	\begin{equation*}
	J_{k\ell} = \int_{\Delta_{k\ell}} |\Delta_Q G(Q,x)| d\sigma(Q).
	\end{equation*}
	
	Using the $(R)_2$ regularity in Lipschitz domain $D_{k\ell}$, we have
	\begin{equation*}
	\begin{aligned}
	J_{k\ell} &\le |\Delta_{k\ell}|^{1/2} \norm{(\nabla^2 G)^*_{D_{k\ell}}(\cdot,x)}_{L^2(\Delta_{k\ell})} \\
	& \le C|\Delta_{k\ell}|^{1/2} \big(\norm{\nabla_{\tan}\nabla G(\cdot,x)}_{L^2(\partial B_{k\ell}\cap \Omega)} +  \norm{\nabla_{\tan} \nabla G(\cdot,x) }_{L^2(\Delta_{k\ell})} \big).
	\end{aligned}
	\end{equation*}
	Since $G(\cdot,x) = \nabla G(\cdot,x) = 0$ on $\partial\Omega$, we see that the last term in the above inequality vanishes. It follows
	\begin{equation*}
	(J_{k\ell})^2 \le C|\Delta_{k\ell}| \int_{\partial B_{k\ell}\cap \Omega} |\nabla^2_y G(y,x)|d\sigma(y).
	\end{equation*}
	
	At this point, one may realize the right-hand  side of the above inequality does not have a good estimate. To overcome this difficulty,  we will use the coarea formula to reduce the surface integral into a volume integral.
	Precisely, for $t\in [1,2]$, let $D_{k\ell}^t = (tB_{k\ell})\cap \Omega$ and $\Delta_{k\ell}^t = (tB_{k\ell}) \cap \partial\Omega$. By the similar argument as above, one can show
	\begin{equation}\label{est.Jkl.t}
	(J_{k\ell}^t )^2 \le C|\Delta_{k\ell}| \int_{\partial (t B_{k\ell})\cap \Omega} |\nabla^2_y G(y,x)|d\sigma(y).
	\end{equation}
	Observing $J_{k\ell} \le J_{k\ell}^t$ and integrating (\ref{est.Jkl.t}) in $t$ over $[1,2]$, one has
	\begin{equation*}
	\begin{aligned}
	(J_{k\ell})^2 &\le \int_1^2 (J_{k\ell}^t)^2 dt \le \frac{C|\Delta_{k\ell}|}{2^{k} \delta(x)} \int_{D_{k\ell}^2} |\nabla_y^2 G(y,x)|^2 dy \\
	& \le  \frac{C|\Delta_{k\ell}|}{2^{3k} \delta(x)^3} \int_{D_{k\ell}^4} |\nabla_y G(y,x)|^2 dy \\
	& \le \frac{C|\Delta_{k\ell}|}{2^{3k} \delta(x)^3} \int_{D_{k\ell}^4} \Big( \frac{C\delta(y)^\e \delta(x)^{1+\e} }{|x-y|^{d-2+2\e}} \Big)^2 dy \\
	& \le C\frac{\delta(x)^2}{2^{2k\e}},
	\end{aligned}
	\end{equation*}
	where we have used the co-area formula in the second inequality, the Caccioppoli inequality in the third inequality (\ref{est.Caccioppoli}) and (\ref{est.DxG.Holder}) in the forth inequality.
	Consequently,
	\begin{equation*}
	\int_{\partial\Omega} |\Delta_Q G(Q,x)| d\sigma(Q) \le \sum_{k = 0}^{\infty} \sum_{\ell = 1}^N J_{k\ell} \le C\delta(x).
	\end{equation*}
	The proof is complete.
\end{proof}

Next, we need to deal with the second integral of (\ref{eq.IntRepre}).
\begin{lemma}\label{lem.DnDeltaG}
	Let $d\ge 4$ and $\Omega$ be a bounded $(\delta,\sigma,R)$-quasiconvex Lipschitz domain with $\delta<\delta_0$. Let
	\begin{equation*}
	u(x) = \int_{\partial \Omega} \frac{\partial}{\partial n} \Delta_Q G(Q,x) f(Q) d\sigma(Q).
	\end{equation*}
	Then
	\begin{equation*}
	|u(x) - u(P)| \le C\delta(x) \norm{\nabla_{\tan} f}_{L^\infty(\partial\Omega)},
	\end{equation*}
	where $P$ is the point on $\partial\Omega$ such that $x\in \Gamma(P)$. 
\end{lemma}

\begin{proof}
	
	The key idea is to use the integration by parts on the boundary. For this purpose, we need to decompose $\partial\Omega$ into the boundaries of finite starlike Lipschitz subdomains and convert the normal derivative to tangential derivatives.
	
	We construct a partition of unity. Let $\varphi_{k\ell}$ be a sequence of smooth functions, with $k \in \N$ and $1\le \ell\le N $, with $N$ depending only on the dimension $d$. Moreover, $\varphi_{k\ell}$ is supported in $B_{k\ell}:= B_{2^k \delta(x)}(Q_{k\ell})$, where $Q_{k\ell} \in \partial\Omega$, $\txt{dist}(x, 2B_{k\ell} ) \simeq 2^k \delta(x)$,
	\begin{equation*}
	\sum_{k,\ell} \varphi_{k\ell}(Q) = 1, \qquad \txt{for any } Q\in \partial\Omega,
	\end{equation*}
	and
	\begin{equation*}
	|\nabla \varphi_{k\ell}| \le \frac{C}{2^k \delta(x)}.
	\end{equation*}
	
	Now write
	\begin{equation*}
	\begin{aligned}
	&u(x) - u(P) \\
	&\qquad = \sum_{k,\ell} \int_{\partial \Omega \cap B_{k\ell} } \frac{\partial}{\partial n} \Delta_Q G(Q,x) (f(Q) - f(P))\varphi_{k\ell}(Q) d\sigma(Q).
	\end{aligned}
	\end{equation*}
	
	For any Lipschitz domains, there exists some $R_0>0$ (depending only on $\Omega$) so that if $r<R_0$, $B_{r}(Q) \cap \Omega$ is a starlike Lipschitz domain for any $Q\in \partial\Omega$. Without loss of generality, we assume all the subdomains $D_{k\ell}: = B_{k\ell} \cap \Omega$ are starlike. Now fix $k$ and $\ell$, and let $x^{k\ell} $ be a suitable ``center" of the starlike domain $D_{k\ell}$. 
	
	Since $D_{k\ell}$ is a starlike Lipschitz domain with ``center'' $x^{k\ell}$, for any given harmonic function $\psi$ in $D_{k\ell}$, we may define a transformation
	\begin{equation*}
	\Psi(y) = \int_0^1 \psi(t(y-x^{k\ell}) + x^{k\ell}) \frac{dt}{t}.
	\end{equation*}
	Then, one may directly verifies that
	\begin{equation*}
	\psi(y) - \psi(x^{k\ell}) = (y-x^{k\ell})\cdot \nabla \Psi(y).
	\end{equation*}
	
	For any $Q\in \partial(D_{k\ell})$, a direct computation shows that
	\begin{equation*}
	\begin{aligned}
	\frac{\partial \psi(Q)}{\partial n} &= \frac{\partial}{\partial n}((y-x^{k\ell})\cdot \nabla \Psi(y)) \Big|_{y=Q} \\
	& = \sum_{i,j} (n_i \partial_j - n_j \partial_i )((y - x^{k\ell})_j \partial_i \Psi) \Big|_{y=Q} + (1-d) \frac{\partial \Psi}{\partial n}(Q).
	\end{aligned}
	\end{equation*}
	Observe that for any fixed $i$ and $j$, $n_i \partial_j - n_j \partial_i$ is a tangential derivative on the boundary $\partial(D_{k\ell})$. Hence, since $\psi(y) = \Delta_y G(y,x)$ is harmonic in $D_{k\ell}$, we have
	\begin{equation*}
	\begin{aligned}
	&\int_{\Delta_{k\ell} } \frac{\partial}{\partial n} \psi(Q) (f(Q) - f(P))\varphi_{k\ell}(Q) d\sigma \\
	& = \sum_{i,j} \int_{\Delta_{k\ell} } (n_i \partial_j - n_j \partial_i )((Q - x^{k\ell})_j \partial_i \Psi) (f(Q) - f(P))\varphi_{k\ell}(Q) d\sigma \\
	& \qquad + (1-d) \int_{\Delta_{k\ell}} \frac{\partial \Psi}{\partial n} (f(Q) - f(P))\varphi_{k\ell}(Q) d\sigma \\
	& = -\sum_{i,j} \int_{\Delta_{k\ell} } (Q - x^{k\ell})_j \partial_i \Psi (n_i \partial_j - n_j \partial_i )\Big( (f(Q) - f(P))\varphi_{k\ell}(Q) \Big) d\sigma \\
	& \qquad + (1-d) \int_{\Delta_{k\ell}} \frac{\partial \Psi}{\partial n} (f(Q) - f(P))\varphi_{k\ell}(Q) d\sigma.
	\end{aligned}
	\end{equation*}
	where $\Delta_{k\ell} = B_{k\ell}\cap \partial\Omega$.

	Now using the result in \cite{V87}, for $1<p<\infty$,
	\begin{equation*}
	\norm{(\nabla \Psi)^*_{D_{k\ell}}}_{L^p(\partial (D_{k\ell}))} \le C \text{diam}(D_{k\ell}) \norm{(\psi)_{D_{k\ell}}^*}_{L^p(\partial (D_{k\ell}))}.
	\end{equation*}
	It follows
	\begin{equation*}
	\begin{aligned}
	&\bigg| \int_{\Delta_{k\ell} } \frac{\partial}{\partial n} \psi(Q) (f(Q) - f(P))\varphi_{k\ell}(Q) d\sigma \bigg| \\
	& \qquad \le C \norm{\nabla_{\tan} f}_{L^\infty(\partial\Omega)} |D_{k\ell}|^{1/2} \bigg( \int_{\partial (D_{k\ell})} |(\psi)_{D_{k\ell}}^*|^2 d\sigma\bigg)^{1/2},
	\end{aligned}
	\end{equation*}
	where we have used the mean value theorem and the fact $|Q-P| \simeq \text{diam}(D_{k\ell})$ for $Q\in D_{k\ell}$.
	
	Now recall that $\psi(y) = \Delta_y G(y,x)$. Similar as $J_{k\ell}$ in the proof of Lemma \ref{lem.DeltaG}, by the $(R)_2$ regularity and the coarea formula, we may derive
	\begin{equation*}
	|D_{k\ell}|^{1/2} \bigg( \int_{\partial (D_{k\ell})} |(\psi)_{D_{k\ell}}^*|^2 d\sigma\bigg)^{1/2} \le C\frac{\delta(x)}{2^{k\e}}.
	\end{equation*}
	Hence
	\begin{equation*}
	|u(x) - u(P)| \le \sum_{k = 0}^{\infty} \sum_{\ell = 1}^{N} \frac{C\delta(x)}{2^{k\e}} \norm{\nabla_{\tan} f}_{L^\infty(\partial\Omega)} \le C\delta(x)\norm{\nabla_{\tan} f}_{L^\infty(\partial\Omega)},
	\end{equation*}
	which ends the proof.
\end{proof}

Now, Theorem \ref{thm.MP} follows readily from the previous two lemmas.

\begin{proof}[Proof of Theorem \ref{thm.MP}]
	For any $x\in \Omega$ and $x\in \Gamma(P), P\in \partial\Omega$, Lemma \ref{lem.DeltaG} and Lemma \ref{lem.DnDeltaG} implies
	\begin{equation}\label{est.uxxx}
	\begin{aligned}
	|u(x) - u(P)| &\le C\delta(x) \big( \norm{g}_{L^\infty(\partial\Omega)}  + \norm{\nabla_{\tan} f}_{L^\infty(\partial\Omega)}\big) \\
	& \le C\delta(x) \norm{g}_{L^\infty(\partial\Omega)}.
	\end{aligned}
	\end{equation}
	The term $\norm{\nabla_{\tan} f}_{L^\infty(\partial\Omega)}$ is not necessary in the last inequality, due to Remark \ref{rmk.D2Dn}. Now, note that if $x\in \Gamma(P)$, for any $y\in B_{\delta(x)/2}(x)$, $y$ is contained in a larger non-tangential cone $\Gamma'(P) \supset \Gamma(P)$, for which (\ref{est.uxxx}) still holds. Since $u(\cdot) - u(P)$ is also a solution, by the interior estimate, one arrives
	\begin{equation*}
	|\nabla u(x)| \le \frac{C}{\delta(x)} \bigg(\fint_{B_{\delta(x)/2}(x)} |u(y) - u(P)|^2 dy\bigg)^{1/2} \le C\norm{g}_{L^\infty(\partial\Omega)}.
	\end{equation*}
	Since $x\in \Omega$ is arbitrary and $\nabla u = g$ on $\partial\Omega$ in the sense of non-tangential limit, we obtain (\ref{est.MP}) as desired.
\end{proof}

\begin{remark}
	Due to Remark \ref{rmk.d23}, our approach for Theorem \ref{thm.MP} provides an alternative proof of the weak maximum principle in arbitrary Lipschitz domains for $d =2$ or $3$.
\end{remark}

\section{Classical Solutions}
In this section, we will show the existence of the classical solution for (\ref{eq.Dp}) if the boundary value $(f,g)\in \WA^{1,2}(\partial\Omega)$ is continuous. For $0\le \alpha < 1$, let
\begin{equation*}
\CA^{1,\alpha}(\partial\Omega) = \{ (f,g) = (\phi,\nabla \phi)|_{\partial\Omega}: \phi\in C^{1,\alpha}_0(\R^d) \}.
\end{equation*}
We stipulate $C^{1,0} = C^1, \CA^{1,0} = \CA^{1}$ and so forth. For $\alpha\in (0,1)$ and $U \subset \R^d$, define
\begin{equation*}
[F]_{C^{\alpha}(U)} := \sup_{x,y\in U} \frac{|F(x) - F(y)|}{|x-y|^\alpha}, \qquad [F]_{C^{1,\alpha}(U)} := \sup_{x,y\in U} \frac{|\nabla F(x) - \nabla F(y)|}{|x-y|^\alpha}.
\end{equation*}
Throughout this section, we assume $\e\in (0,1)$ and $\delta_0>0$ is given by Corollary \ref{cor.DuCe} (or Theorem \ref{thm.Gxy}). Note that $\delta_0$ depends only on $d,\e$ and the Lipschitz constant.

Let $(f,g)\in \CA^{1,\alpha}(\partial\Omega)$ and $h = n\cdot g$. Consider the equation
\begin{equation}\label{eq.biharmonic}
\left\{
\begin{aligned}
&\Delta^2 u = 0, \quad \txt{in } \Omega, \\
&u = f,\  \frac{\partial}{\partial n}u = h, \quad \txt{on } \partial\Omega.
\end{aligned}
\right.
\end{equation}
In this section, we will show that if $\Omega$ is a quasiconvex domain with $\delta<\delta_0$, $0\le \alpha<\e$ and $(f,g)\in \CA^{1,\alpha}(\partial\Omega)$, then $u\in C^{1,\alpha}(\overline{\Omega})$.

To this end, we fix an arbitrary point $P_0\in \partial\Omega$ and would like to subtract a linear function from $u$ so that the resulting function and its gradient both vanish at $P_0$. We first assume $\alpha>0$. By our definition, $(f,g) = (\phi,\nabla \phi)$ on $\partial\Omega$ for some $\phi\in C^{1,\alpha}_0(\R^d)$. Then, define a linear function
\begin{equation*}
L(x) = \phi(P_0) + \nabla \phi(P_0) \cdot (x-P_0).
\end{equation*}
Now, if we set
\begin{equation*}
\widetilde{u}(x) = u(x) - L(x),
\end{equation*}
then $\widetilde{u}$ satisfies
\begin{equation*}
\left\{
\begin{aligned}
&\Delta^2 \widetilde{u} = 0, \quad \txt{in } \Omega, \\
&\widetilde{u} = \widetilde{f},\  \frac{\partial}{\partial n}u = \widetilde{h}, \quad \txt{on } \partial\Omega,
\end{aligned}
\right.
\end{equation*}
where $\widetilde{f} = \phi-L$ and $\widetilde{h} = n \cdot ( \nabla \phi -\nabla L)$ on $\partial\Omega$. Note that $\widetilde{h}$ is well-defined as the normal $n(Q)$ exists for a.e. $Q\in \partial\Omega$. Using the observations $(\phi-L)(P_0) = 0$ and $\nabla(\phi-L)(P_0) = 0$, we have
\begin{equation*}
|\widetilde{f}(Q)| \le C|Q-P_0|^{1+\alpha} [\phi]_{C^{1,\alpha}(\R^d)},
\end{equation*}
and
\begin{equation}\label{est.ha}
|\widetilde{h}(Q)| \le C|Q-P_0|^\alpha [\phi]_{C^{1,\alpha}(\R^d)}.
\end{equation}

\begin{theorem}\label{thm.tu.C1a}
	Let $d\ge 4$ and $\Omega$ be a bounded $(\delta,\sigma,R)$-quasiconvex Lipschitz domain with $\delta<\delta_0$. Let $\widetilde{u}$ be defined as above. Then, for any $x\in \Gamma(P_0)$
	\begin{equation*}
	|\nabla \widetilde{u}(x)| + \delta(x)^{-1} |\widetilde{u}(x)|  \le C\delta(x)^\alpha [\phi]_{C^{1,\alpha}(\R^d)}.
	\end{equation*}
	where $\alpha\in (0,\e)$ and $C$ depends only on $d, \e, \alpha$ and $\Omega$.
\end{theorem}

Recall the integral representation for the solution $\widetilde{u}$
\begin{equation}\label{eq.tu.int}
\widetilde{u}(x) = \int_{\partial \Omega} \frac{\partial}{\partial n} \Delta_Q G(Q,x) \widetilde{f}(Q) d\sigma(Q) + \int_{\partial \Omega} \Delta_Q G(Q,x) \widetilde{h}(Q) d\sigma(Q).
\end{equation}
The following are improved estimates of Lemma \ref{lem.DeltaG} and Lemma \ref{lem.DnDeltaG}.
\begin{lemma}\label{lem.DG.alpha}
	Let $d\ge 4$ and $\Omega$ be a bounded $(\delta,\sigma,R)$-quasiconvex Lipschitz domain with $\delta<\delta_0$. Then
	\begin{equation*}
	\int_{\partial\Omega} |\Delta_Q G(Q,x)|  |x-Q|^{\alpha}d\sigma(Q) \le C\delta(x)^{1+\alpha},
	\end{equation*}
	where $\alpha\in (0,\e)$ and $C$ depends only on $d, \e, \alpha$ and $\Omega$.
\end{lemma}

\begin{proof}
	For each $x\in \Omega$, let $D_{k\ell}$ and $\Delta_{k\ell}$ be the same as Lemma \ref{lem.DeltaG}. Observe that if $Q\in \Delta_{k\ell}$ and $y\in D_{k\ell}$,
	\begin{equation*}
	|x-Q| \simeq |x-y| \simeq 2^k \delta(x).
	\end{equation*}
	Let
	\begin{equation*}
	\widetilde{J}_{k\ell} = \int_{\Delta_{k\ell}} |\Delta_Q G(Q,x)|  |x-Q|^{\alpha}d\sigma(Q).
	\end{equation*}
	By the similar argument as Lemma \ref{lem.DeltaG}, we have
	\begin{equation*}
	\begin{aligned}
	(\widetilde{J}_{k\ell})^2 &\le \frac{C|\Delta_{k\ell}|}{2^{3k} \delta(x)^3} \int_{D_{k\ell}^4} \Big( \frac{C\delta(y)^\e \delta(x)^{1+\e} }{|x-y|^{d-2+2\e}} \Big)^2 (2^k\delta(x))^{2\alpha} dy \\
	& \le C\frac{\delta(x)^{2+2\alpha}  }{2^{2k(\e-\alpha)}}.
	\end{aligned}
	\end{equation*}
	Now if $\alpha<\e$, we may obtain the desired estimate by taking square root and summing over $k$ and $\ell$.
\end{proof}

\begin{lemma}\label{lem.DnDG.alpha}
	Let $d\ge 4$ and $\Omega$ be a bounded $(\delta,\sigma,R)$-quasiconvex Lipschitz domain with $\delta<\delta_0$. Let
	\begin{equation*}
	u(x) = \int_{\partial \Omega} \frac{\partial}{\partial n} \Delta_Q G(Q,x) \widetilde{f}(Q) d\sigma(Q).
	\end{equation*}
	Then, if $x\in \Gamma(P_0)$,
	\begin{equation*}
	|u(x)| \le C\delta(x)^{1+\alpha} [\phi]_{C^{1,\alpha}(\R^d)},
	\end{equation*}
	where $\alpha\in (0,\e)$ and $C$ depends only on $d, \e, \alpha$ and $\Omega$.
\end{lemma}

\begin{proof}
	By our construction $\widetilde{f} = \phi - L$ (defined in the entire $\R^d$), and the facts $(\phi-L)(P_0) = 0$ and $\nabla(\phi-L)(P_0) = 0$, we know
	\begin{equation}\label{est.tf}
	|\nabla \widetilde{f}(Q) |+ |Q-P_0|^{-1} |\widetilde{f}(Q)|  \le C|Q-P_0|^{\alpha} [\phi]_{C^{1,\alpha}(\R^d)}.
	\end{equation}
	For each $x\in \Gamma(P_0)$, let $D_{k\ell}$, $\Delta_{k\ell}$ and $\varphi_{k\ell}$ be the same as Lemma \ref{lem.DnDeltaG}. Note that for $Q\in \Delta_{k\ell}$, $|Q-P_0| \le C 2^{k} \delta(x)$. Therefore,
	by (\ref{est.tf}) and the same argument as Lemma \ref{lem.DnDeltaG}, one has
	\begin{equation*}
	\begin{aligned}
	&\bigg| \int_{\Delta_{k\ell} } \frac{\partial}{\partial n} \psi(Q) \widetilde{f}(Q) \varphi_{k\ell}(Q) d\sigma \bigg| \\
	& \qquad \le C (2^k \delta(x))^\alpha [\phi]_{C^{1,\alpha}(\R^d)} |D_{k\ell}|^{1/2} \bigg( \int_{\partial (D_{k\ell})} |(\psi)_{D_{k\ell}}^*|^2 d\sigma\bigg)^{1/2} \\
	& \qquad \le C\frac{\delta(x)^{1+\alpha}}{2^{k(\e-\alpha)}},
	\end{aligned}
	\end{equation*}
	where $\psi(Q) = \Delta_Q G(Q,x)$. This implies the desired estimate for $u(x)$ by summing over $k$ and $\ell$, provided $\alpha<\e$.
\end{proof}

\begin{proof}[Proof of Theorem \ref{thm.tu.C1a}]
	The estimate of $|\widetilde{u}(x)|$ follows readily from the representation formula (\ref{eq.tu.int}), Lemma \ref{lem.DG.alpha}, Lemma \ref{lem.DnDG.alpha}, and (\ref{est.ha}). Now, note that if $x\in \Gamma(P_0)$, then $B_{\delta(x)/4}(x)$ is contained in a larger non-tangential cone $\Gamma'(P_0)$, in which the estimate for $\widetilde{u}$ still holds. Then, the interior estimate gives
	\begin{equation*}
	|\nabla \widetilde{u}(x)| \le C\delta(x)^{-1} \sup_{y\in B_{\delta(x)/4} (x)} |\widetilde{u}(y)| \le C\delta(x)^\alpha [\phi]_{C^{1,\alpha}(\R^d)}.
	\end{equation*}
	This completes the proof.
\end{proof}

The following is the main result of this section.
\begin{theorem}\label{thm.C1a}
	Let $d\ge 4$ and $\Omega$ be a bounded $(\delta,\sigma,R)$-quasiconvex Lipschitz domain with $\delta<\delta_0$. Let $u$ be the solution of (\ref{eq.biharmonic}). If $0<\alpha<\e$ and $ (f,g) = (\phi,\nabla\phi)\in \CA^{1,\alpha}(\partial\Omega)$ with $\phi\in C^{1,\alpha}(\R^d)$, then $u\in C^{1,\alpha}(\overline{\Omega})$. Moreover,
	\begin{equation}\label{est.C1a.phi}
	[\nabla u]_{C^\alpha(\Omega)} \le C[\phi]_{C^{1,\alpha}(\R^d)},
	\end{equation}
	where $C$ depends only on $d,\alpha$ and $\Omega$.
\end{theorem}

\begin{proof}
	We first show that $\nabla u$ is $C^\alpha$-H\"{o}lder continuous on the boundary along the non-tangential directions. Let $\widetilde{u}$ be defined as before with some $P_0\in \partial\Omega$. Note that $\nabla \widetilde{u}(P_0) = \nabla u(P_0) - \nabla L(P_0) = 0$ and that $\nabla L(x)$ is constant. Thus,
	\begin{equation*}
	|\nabla \widetilde{u}(x)| = |\nabla u(x) - \nabla L(x) - \nabla u(P_0) + \nabla L(P_0)| = |\nabla u(x)  - \nabla u(P_0)|.
	\end{equation*}
	Hence, Theorem \ref{thm.tu.C1a} implies
	\begin{equation}\label{est.C1a.nontan}
	|\nabla u(x) - \nabla u(P_0)| \le C\delta(x)^\alpha [\phi]_{C^{1,\alpha}(\R^d)}.
	\end{equation}
	Apparently, (\ref{est.C1a.nontan}) holds for any $P_0\in \partial\Omega$ and $x\in \Gamma(P_0)$.
	
	On the other hand, by our assumption on $\phi$, $\nabla u(Q) = \nabla \phi(Q)$ is obviously $C^\alpha$-H\"{o}lder continuous for $Q\in \partial\Omega$, i.e., for any $P,Q\in \partial\Omega$
	\begin{equation}\label{est.bdry.C1a}
	|\nabla u(P) - \nabla u(Q)| \le C|P-Q|^\alpha [\phi]_{C^{1,\alpha}(\R^d)}.
	\end{equation}
	
	Now, consider any two points $x,y\in \Omega$ and $x\in \Gamma(P)$ and $y\in \Gamma(Q)$. Without loss of generality, we assume $\delta(x) \ge \delta(y)$. It suffices to estimate $|\nabla u(x) - \nabla u(y)|$, which will be done in two cases as follows. 
	
	Case 1: $|x-y| \le \delta(x)/4$. Let $\widetilde{u}$ be constructed as previously with respect to the point $P$ (instead of $P_0$). Since $x\in \Gamma(P)$, $B_{\delta(x)/2}(x)$ is contained in a larger non-tangential cone $\Gamma'(P)$ for which Theorem \ref{thm.tu.C1a} still holds. Thus, using the interior $C^{1,\alpha}$ estimate for $\widetilde{u}$ and Theorem \ref{thm.tu.C1a}, we obtain
	\begin{equation}\label{est.tuxy}
	\begin{aligned}
	|\nabla \widetilde{u}(x) - \nabla \widetilde{u}(y) | &\le C\bigg( \frac{|x-y|}{\delta(x)} \bigg)^{\alpha} \bigg( \fint_{B_{\delta(x)/2} (x) }  |\nabla \widetilde{u}|^2 \bigg)^{1/2} \\
	& \le C|x-y|^\alpha [\phi]_{C^{1,\alpha}(\R^d)}.
	\end{aligned}
	\end{equation}
	Recall that $\widetilde{u} = u - L$ and $\nabla L$ is constant. Thus, (\ref{est.tuxy}) gives
	\begin{equation*}
	|\nabla u(x) - \nabla u(y)|  \le C|x-y|^\alpha [\phi]_{C^{1,\alpha}(\R^d)}.
	\end{equation*}
	
	Case 2: $|x-y| \ge \delta(x)/4$. In this case,
	\begin{equation*}
	\begin{aligned}
	|P-Q| &\le |P-x| + |x-y| + |y-Q| \\
	&\le (1+\beta) \delta(x) + |x-y| + (1+\beta)\delta(y) \\
	&\le (8(1+\beta)+1) |x-y|.
	\end{aligned}
	\end{equation*}
	It follows from (\ref{est.C1a.nontan}) and (\ref{est.bdry.C1a}) that
	\begin{equation*}
	\begin{aligned}
	|\nabla u(x) - \nabla u(y)| &\le |\nabla u(x) - \nabla u(P)| + |\nabla u(P) - \nabla u(Q)| + |\nabla u(Q) - \nabla u(y)| \\
	& \le C\delta(x)^\alpha [\phi]_{C^{1,\alpha}(\R^d)} + C|P-Q|^\alpha [\phi]_{C^{1,\alpha}(\R^d)} + C\delta(y)^\alpha [\phi]_{C^{1,\alpha}(\R^d)} \\
	& \le C|x-y|^\alpha [\phi]_{C^{1,\alpha}(\R^d)},
	\end{aligned}
	\end{equation*}
	The proof hence is complete.
\end{proof}

As a corollary, we can deal with the endpoint case $\alpha=0$ by the weak maximum principle and an approximation argument.

\begin{corollary}
	Let $d\ge 4$ and $\Omega$ be a bounded $(\delta,\sigma,R)$-quasiconvex Lipschitz domain with $\delta<\delta_0$. If $u$ is the solution of (\ref{eq.biharmonic}) with $ (f,g) \in \CA^{1}(\partial\Omega)$, then $u\in C^1(\overline{\Omega})$. Moreover, for any $Q\in \partial\Omega$,
	\begin{equation*}
	\lim_{\Omega \ni x\to Q } u(x) = f(Q), \quad \text{and} \quad \lim_{\Omega \ni x\to Q } \nabla u(x) = g(Q).
	\end{equation*}
\end{corollary}

\begin{proof}
	Let $(f,g) = (\phi,\nabla \phi)|_{\partial\Omega} \in \CA^{1}(\partial\Omega)$ for some $\phi\in C_0^1(\R^d)$. For any given $\rho>0$, we may construct $\phi_\rho \in C^{1,\alpha}_0(\R^d)$ for some $\alpha\in (0,\e)$ so that
	\begin{equation}\label{est.phi.rho}
	\norm{\phi - \phi_\rho}_{L^\infty(\R^d)} \le \rho, \qquad \norm{\nabla \phi - \nabla \phi_\rho}_{L^\infty(\R^d)} \le \rho.
	\end{equation}
	Let $(f_\rho,g_\rho) = (\phi_\rho,\nabla \phi_\rho)|_{\partial\Omega} \in \CA^{1,\alpha}(\partial\Omega)$. 
	Thus, (\ref{est.phi.rho}) implies 
	\begin{equation*}
	\norm{f - f_\rho}_{L^\infty(\R^d)} \le \rho, \qquad \norm{g - g_\rho}_{L^\infty(\R^d)} \le \rho.
	\end{equation*}
	
	Consider the biharmonic equation with data $(f_\rho,g_\rho)$
	\begin{equation}
	\left\{
	\begin{aligned}
	&\Delta^2 u_\rho = 0, \quad \txt{in } \Omega, \\
	&u_\rho = f_\rho,\  \frac{\partial}{\partial n}u_\rho = h_\rho, \quad \txt{on } \partial\Omega,
	\end{aligned}
	\right.
	\end{equation}
	where $h_\rho = n\cdot g_\rho$. Then, obviously, the difference $u - u_\rho$ satisfies
	\begin{equation}
	\left\{
	\begin{aligned}
	&\Delta^2 (u-u_\rho) = 0, \quad \txt{in } \Omega, \\
	&u - u_\rho = f - f_\rho,\  \frac{\partial}{\partial n}(u-u_\rho) = n\cdot (g - g_\rho), \quad \txt{on } \partial\Omega,
	\end{aligned}
	\right.
	\end{equation}
	where $(f-f_\rho, g-g_\rho) \in \WA^{1,\infty}(\partial\Omega)$. Now, by the weak maximum principle, we have
	\begin{equation*}
	\norm{\nabla u - \nabla u_\rho}_{L^\infty(\Omega)} \le C\norm{g - g_\rho} \le C\rho.
	\end{equation*}
	Since the domain $\Omega$ is bounded, the above estimate also gives
	\begin{equation*}
	\norm{u-u_\rho}_{L^\infty(\Omega)} \le C\rho,
	\end{equation*}
	where $C$ is independent of $\rho$.
	
	On the other hand, by Theorem \ref{thm.C1a}, for any given $Q\in \partial\Omega$,
	\begin{equation*}
	\lim_{\Omega \ni x\to Q } u_\rho(x) = f_\rho(Q), \quad \text{and} \quad \lim_{\Omega \ni x\to Q } \nabla u_\rho(x) = g_\rho(Q).
	\end{equation*}
	Therefore,
	\begin{equation*}
	\begin{aligned}
	\limsup_{\Omega \ni x\to Q} |u(x) - f(Q)| & \le \norm{u - u_\rho}_{L^\infty(\Omega)} + \limsup_{\Omega \ni x\to Q} |u_\rho(x) - f_\rho(Q)| + |f(Q) - f_\rho(Q)| \\
	& \le C\rho.
	\end{aligned}
	\end{equation*}
	Since $\rho>0$ is arbitrary, we may let $\rho \to 0$ and obtain
	\begin{equation*}
	\lim_{\Omega \ni x\to Q } u(x) = f(Q).
	\end{equation*}
	The proof for $\nabla u(x)$ is the same.
\end{proof}

\begin{remark}
	In view of Remark \ref{rmk.d23}, the above classical solution may also be obtained in arbitrary Lipschitz domains for $d = 2$ or $3$. In particular, for $d=2$, we may even show (\ref{est.C1a.phi}) with $\alpha \in (0,\frac{1}{2} +\e)$, instead of $\alpha\in (0,\e)$. This is due to a better reverse H\"{o}lder inequality. Actually, by the same argument in Remark \ref{rmk.d23}, we may use the $(R)_{2+\e}$ regularity to show the reverse H\"{o}lder inequality with $p = 4+\e$. By the Sobolev embedding theorem, this implies the Green's function is actually of $C^{1,\frac{1}{2} + \e}$ (excluding the poles), which yields Lemma \ref{lem.DG.alpha} and Lemma \ref{lem.DnDG.alpha} with $\alpha\in (0,\frac{1}{2}+\e)$.
\end{remark}

\textbf{Acknowledgment.} The author is supported in part by National Science Foundation grant DMS-1600520.

\bibliographystyle{plain}
\bibliography{mybib}
\end{document}